\documentclass[12pt, twoside, final]{article}

\usepackage{amsmath, amsthm, amssymb}
\usepackage{times}

\usepackage[english]{babel}
\usepackage[utf8]{inputenc}
\usepackage{mathrsfs}                     
\usepackage{bbm}                          
\usepackage[all]{xy}                      
\usepackage{graphicx}                     
\usepackage{color}
\usepackage{url}
\usepackage{import}                       
\usepackage{todonotes}                    
\usepackage{algpseudocode}                
\usepackage[style=alphabetic, minnames=3,
            doi=false, url=false, isbn=false,
            firstinits=true,
            sortcites=true,
            backend=bibtex]{biblatex}      
\usepackage{enumerate}                     
\usepackage{paralist}                      

\def\titlerunning#1{\gdef\titrun{#1}}
\makeatletter
\def\author#1{\gdef\autrun{\def\and{\unskip, }#1}\gdef\@author{#1}}
\def\address#1{{\def\and{\\\hspace*{18pt}}\renewcommand{\thefootnote}{}%
\footnote {#1}}%
\markboth{\autrun}{\titrun}}
\makeatother
\def\email#1{e-mail: #1}
\def\subjclass#1{{\renewcommand{\thefootnote}{}%
\footnote{\emph{Mathematics Subject Classification (2010):} #1}}}
\def\keywords#1{\par\medskip
\noindent\textbf{Keywords.} #1}

\theoremstyle{plain}
\newtheorem{theorem}[subsection]{Theorem}
\newtheorem{lemma}[subsection]{Lemma}
\newtheorem{corollary}[subsection]{Corollary}

\newtheorem{proposition}[subsection]{Proposition}
\newtheorem*{theorem*}{Theorem}

\theoremstyle{definition}
\newtheorem{definition}[subsection]{Definition}
\newtheorem{remark}[subsection]{Remark}

\newtheorem*{remark*}{Remark}
\newtheorem*{example*}{Example}
\newtheorem*{openquestion*}{Open Question}

\numberwithin{equation}{section}

\def\de{\delta} 
\def\ep{\varepsilon} 
 
\def\et{\eta}

\def\la{\lambda} 
 
\def\si{\sigma} 
\def\ta{\tau} 
\def\ph{\varphi} 
\def\ch{\chi} 
\def\ps{\psi} 
\def\om{\omega} 
 
\def\De{\Delta} 
\def\Th{\Theta}

\def\Om{\Omega}
 
\def\o{\circ} 
\def\i{^{-1}} 
\def\x{\times}
\def\p{\partial}

\def\R{{\mathbb R}}

\def\exp{\operatorname{exp}}

\let\on=\operatorname

\let\wh=\widehat
\let\ol=\overline

\let\mb=\mathbb
\let\mc=\mathcal
\let\mf=\mathfrak

\newcommand{\ud}{\,\mathrm{d}}

\begin{document}

\baselineskip=17pt

\titlerunning{On Completeness of Groups of Diffeomorphisms}

\title{On Completeness of Groups of Diffeomorphisms}

\author{Martins Bruveris
\and 
Fran\c{c}ois-Xavier Vialard}

\date{\today}

\maketitle

\address{M. Bruveris, Department of Mathematics, Brunel University London, Uxbridge UB8 3PH, United Kingdom; \email{martins.bruveris@brunel.ac.uk}
\and
F.-X. Vialard, Université Paris-Dauphine,
Place du Maréchal de Lattre de Tassigny,
75775 Paris CEDEX 16, France;
\email{vialard@ceremade.dauphine.fr}
}


\begin{abstract}
We study completeness properties of the Sobolev diffeomorphism groups $\mc D^s(M)$  endowed with strong right-invariant Riemannian metrics when $M$ is $\R^d$ or a compact manifold without boundary. We prove that for $s > \dim M/2 + 1$, the group $\mc D^s(M)$  is geodesically and metrically complete and any two diffeomorphisms in the same component can be joined by a minimal geodesic. We then present the connection between the Sobolev diffeomorphism group and the \emph{large deformation matching} framework in order to apply our results to diffeomorphic image matching.

\end{abstract}

\section{Introduction}

The interest in Riemannian geometry of diffeomorphism groups started with \cite{Arnold1966}, where it was shown that Euler's equations, describing the motion of an ideal, incompressible fluid, can be regarded as geodesic equations on the group of volume-preserving diffeomorphisms. The corresponding Riemannian metric is the right-invariant $L^2$-type metric. This was used in \cite{Ebin1970} to show the local well-posedness of Euler's equations in three and more dimensions.  Also following \cite{Arnold1966}, the curvature of the Riemannian metric was connected in \cite{Misiolek1993,Shkoller1998,Preston2004} to stability properties of the fluid flow. The Fredholmness of the Riemannian exponential map was used in \cite{Misiolek2010} to show that large parts of the diffeomorphism group is reachable from the identity via minimising geodesics.

Other equations that have been recognised as geodesic equations on the diffeomorphism groups include the Camassa--Holm equation \cite{Holm1993}, the Korteweg--de Vries equation \cite{Ovsienko1987, Segal1991}, the quasigeostrophic equation \cite{Ebin2012, Ebin2015}, the equations of a barotropic fluid \cite{Preston2013} and others; see \cite{Vizman2008,Bauer2014} for an overview. In \cite{Escher2011}, the Degasperis-Procesi equation is identified as being a geodesic equation for a particular right-invariant connection on the diffeomorphism group. 

\subsection*{Right-invariant Sobolev metrics} Let $M$ be either $\R^d$ or a compact manifold without boundary of dimension $d$. The group $\mc D^s(M)$, with $s > d/2+1$, consists of all $C^1$-diffeomorphisms of Sobolev regularity $H^s$. It is well-known that $\mc D^s(M)$ is a smooth Hilbert manifold and a topological group \cite{Inci2013}. Right-invariant Sobolev $H^r$-metrics on diffeomorphism groups can thus be described using two parameters: the order $r$ of the metric and the regularity $s$ of the group. Obviously one requires $r \leq s$ for the metric to be well-defined.

As far as the behaviour of Sobolev metrics is concerned, the regularity $s$ of the group is less important that the order $r$ of the metric. Many properties like smoothness of the geodesic spray, (non-)vanishing of the geodesic distance, Fredholmness of the exponential map are not present for $H^r$-metrics with $r$ small and then ``emerge'' at a certain critical value of $r$. For some, like the Fredholmness properties of the exponential map, the critical value is independent of the dimension of $M$, in other cases the independence is conjectured and in yet others, like the completeness results in this paper, the critical value does depend on the dimension. The range of admissible values for $s$ is in each case usually an interval bounded from below with the lower bound depending on $r$.

The study of Sobolev metrics is complicated by the fact that, for a given order $r$, there is no canonical $H^r$-metric, just like there is no canonical $H^r$-inner product on the space $H^r(M,\R)$. The topology is canonical, but the inner product is not. For $r \in \mb N$, a class of ``natural'' inner products can be defined using the intrinsic differential operations on $M$. They are of the form
\begin{equation}
\label{eq_def_metric_L}
\langle u , v \rangle_{H^r} = \int_M \langle u, L v \rangle \ud \mu\,,
\end{equation}
where $L$ is a positive, invertible, elliptic differential operator of order $2r$. For (possibly) non-integer orders, the most general family of inner products is given by pseudodifferential operators $L \in OPS^{2r}$ of order $2r$ within a certain symbol class. The corresponding Riemannian metric is
\[
G_\ph(X_\ph, Y_\ph) = \int_{M} \left\langle X_\ph \o \ph\i, L( Y_\ph \o \ph\i) \right\rangle \ud \mu\,,
\]
and it can be represented by the operator $L_\ph = R_{\ph\i}^\ast \o L \o R_{\ph\i}$ with $R_\ph X = X \o \ph$ denoting right-translation by $\ph$. Note however, that $\ph$ is not smooth, but only in $\mc D^s(M)$ and thus $L_\ph$ is not a pseudodifferential operator with a smooth symbol any more. Pseudodifferential operators with symbols in Sobolev spaces were studied for example in \cite{Adams1986, Adams1986b, Beals1984, Lannes2006}, but technical difficulties still remain.

\subsection*{Strong Sobolev metrics}
Historically most papers dealt with right-inva\-riant Sobolev metrics on diffeomorphism groups in the weak setting, that is one considered $H^r$-metrics on $\mc D^s(M)$ with $s > r$; a typical assumption is $s > 2r+d/2+1$, in order to ensure that $Lu$ is still $C^1$-regular. The disconnect between the order of the metric and the regularity of the group arose, because one was mostly interested in $L^2$ or $H^1$-metrics, but $\mc D^s(M)$ is a Hilbert manifold only when $s > d/2+1$. It was however noted already in \cite{Ebin1970} and again in \cite{Misiolek2010}, that the $H^s$-metric is well-defined and, more importantly, smooth on $\mc D^s(M)$, for integer $s$ when the inner product is  defined in terms of a differential operator as in \eqref{eq_def_metric_L}. The smoothness of the metric is not obvious, since it is defined via
\[
G_\ph(X_\ph, Y_\ph) = \langle X_\ph \o \ph\i, Y_\ph\o\ph\i \rangle_{H^s}
\]
and the definition uses the inversion, which is only a continuous, but not a smooth operation on $\mc D^s(M)$.

Higher order Sobolev metrics have been studied recently on diffeomorphism groups of the circle \cite{Constantin2003}, of the torus \cite{Kappeler2008} and of general compact manifolds \cite{Misiolek2010}. The sectional curvature of such metrics was analysed in \cite{Khesin2013b} and in \cite{Bauer2011b,Bauer2012d} the authors considered Sobolev metrics on the space of immersions, which contains the diffeomorphism group as a special case.

\subsection*{Diffeomorphic image matching}
Another application of strong Sobolev metrics on the diffeomorphism group is the field of computational anatomy and diffeomorphic image matching \cite{Grenander1998}. Given two images, represented by scalar functions $I, J: \R^d \to \R$, diffeomorphic image registration is the problem of solving the minimization problem
\[
\mc J(\ph) = \on{dist}(\on{Id}, \ph) + S(I \o \ph\i, J)\,,
\]
over a suitable group of diffeomorphisms; here $S$ is a similarity measure between images, for example the $L^2$-norm, and $\on{dist}$ is a distance between diffeomorphisms \cite{Beg2005}. In the large deformation matching framework this distance is taken to be the geodesic distance of an underlying right-invariant Riemannian metric on the diffeomorphism group. Thus Sobolev metrics comprise a natural family of metrics to be used for diffeomorphic image registration.

\subsection*{Completeness}
The contributions of this paper are twofold. First we want to show that strong, smooth Sobolev metrics on $\mc D^s(M)$ are geodesically and metrically complete and that there exist minimizing geodesics between any two diffeomorphisms. We recall here that the Hopf--Rinow theorem is not valid in infinite dimensions, namely Atkin gives in \cite{Atkin1975} an example of a geodesically complete Riemannian manifold where  the exponential map is not surjective. For the Sobolev diffeomorphism group with $s>d/2+1$, the best known result can be found in \cite[Thm. 9.1]{Misiolek2010} which is an improvement of the positive result of Ekeland \cite{Ekeland1978}.

Geodesic completeness was shown for the diffeomorphism group of the circle in \cite{Escher2014b} and in weaker form on $\R^d$ in \cite{Trouve2005a} and \cite{Michor2013}. Metric completeness and existence of minimizing geodesics in the context of groups of Sobolev diffeomorphisms and its subgroups is---as far as we know---new. We prove the following theorem:

\begin{theorem*}
Let $M$ be $\R^d$ or a closed manifold and $s > d/2+1$. If $G^s$ is a smooth, right-invariant Sobolev-metric of order $s$ on $\mc D^s(M)$, then
\begin{enumerate}
\item
$(\mc D^s(M), G^s)$ is geodesically complete;
\item
$(\mc D^s(M)_0, \on{dist}^s)$ is a complete metric space;
\item
Any two elements of $\mc D^s(M)_0$ can be joined by a minimizing geodesic.
\end{enumerate}
\end{theorem*}

We expect that the same methods of proof can also be applied to the subgroups $\mc D^s_\mu(M)$ and $\mc D^s_\om(M)$ of diffeomorphisms preserving a volume form $\mu$ or a symplectic structure $\om$.

The crucial ingredient in the proof is showing that for each $t$ the flow map
\begin{equation}
\label{eq:intro_flow}
\on{Fl}_t : L^1(I, \mf X^s(M)) \to \mc D^s(M)\,,
\end{equation}
assigning a vector field its flow at time $t$, exists and is continuous; see Sec.~\ref{sec:flow_def} for definitions. The existence was known for vector fields in $C(I, \mf X^s(M))$ and the continuity as a map into $\mc D^{s'}$ for $s'< s$ was shown in \cite{Inci2012}. We extend the existence result to vector fields that are $L^1$ in time and show continuity with respect to the manifold topology. The flow map allows us to identify the space of $H^1$-paths with the space of right-trivialized velocities,
\[
\mc D^s(M) \x L^2(I, \mf X^s(M)) \xrightarrow{\cong} H^1(I,\mc D^s(M)),
\quad
(\ph_0, u) \mapsto \left( t \mapsto \on{Fl}_t(u) \o \ph_0 \right)\,.
\]
The inverse map of the identification is given by $H^1(I,\mc D^s) \ni \ph \mapsto \left( \ph(0), \p_t \ph \o \ph\i \right)$.
Since $L^2(I,\mf X^s(M))$ is a Hilbert space, we can use variational methods to show the existence of minimizing geodesics.

In order to show metric completeness, we derive, in the case $M = \R^d$, the following estimate on the geodesic distance,
\[
\| \ph - \ps \|_{H^s} \leq C \on{dist}^s(\ph, \ps)\,,
\]
which is valid on a bounded metric $\on{dist}^s$-ball. In other words, the identity map between the two metric spaces
\[
\on{Id} : \left( \mc D^s(\R^d), \| \cdot \|_{H^s} \right) \to
\left( \mc D^s(\R^d), \on{dist}^s \right)
\]
is locally Lipschitz continuous. For compact manifolds we show a similar inequality in coordinate charts. The Lipschitz continuity implies that a Cauchy sequence for $\on{dist}^s$ is a Cauchy sequence for $\| \cdot \|_{H^s}$, thus giving us a candidate for a limit point. One then proceeds to show that the limit point lies in the diffeomorphism group and that the sequence converges to it with respect to the geodesic distance.

\subsection*{Applications to image matching}
The second contribution concerns the groups of diffeomorphisms introduced by Trouv\'e \cite{Trouve1998,Trouve2005a} for diffeomorphic image matching in the large deformation framework \cite{Beg2005}. In this framework one chooses a Hilbert space $\mc H$ of vector fields on $\R^d$ with a norm that is stronger than the uniform $C^1_b$-norm\footnote{The $C^1_b$-norm is the supremum norm on the vector field and the first derivative, $\| u\|_{C^1_b} = \| u \|_\infty + \| Du \|_\infty$.}, i.e., $\mc H \hookrightarrow C^1_b$ and considers the group $\mc G_{\mc H}$ of all diffeomorphisms, that can be generated as flows of vector fields in $L^2(I,\mc H)$, $I$ being a compact interval.

When $s>d/2+1$ the Sobolev embedding theorem shows that $H^s \hookrightarrow C^1_b$, allowing us to consider the group $\mc G_{H^s}$ as a special case of the construcion by Trouvé. It is not difficult to show, for $t$ fixed, the existence of the flow as a map
\[
\on{Fl}_t : L^2(I, \mc H) \to \on{Diff}^1(\R^d)
\]
into the space of $C^1$-diffeomorphisms. Thus we can view the existence of the flow map in the sense \eqref{eq:intro_flow} as a regularity result when $\mc H = H^s$. With the help of this regularity result we are able to show the following:

\begin{theorem*}
Let $s > d/2+1$. Then
$\mc G_{H^s} = \mc D^s(\R^d)_0$\,.
\end{theorem*}
Here $\mc D^s(\R^d)_0$ is the connected component of the identity. This means that, if we choose $\mc H$ to be a Sobolev space, then the framework of Trouvé constructs the classical groups of Sobolev diffeomorphisms. As a consequence we obtain that $\mc G_{H^s}$ is a topological group and that the paths solving the image registration problem are smooth. We also obtain using the proximal calculus on Riemannian manifolds \cite{Azagra2005} that Karcher means of $k$ diffeomorphisms -- and more generally shapes -- are unique on a dense subset of the $k$-fold product $ \mc D^s \x \dots \x \mc D^s$. 

\section{The group $\mc D^s(\R^d)$}

The Sobolev spaces $H^s(\R^d)$ with $s \in \R$ can be defined in terms of the Fourier transform
\[ 
\mc F f(\xi) = (2\pi)^{-n/2} \int_{\R^n} e^{-i \langle  x,\xi\rangle} f(x) \ud x\,,
\]
and consist of $L^2$-integrable functions $f$ with the property that $(1+|\xi|^2)^{s/2} \mc F f$ is $L^2$-integrable as well. An inner product on $H^s(\R^d)$ is given by
\[
\langle f, g \rangle_{H^s} = \mathfrak{Re}\int_{\R^d} (1 + |\xi|^2)^s \mc F f(\xi) \overline{\mc F g(\xi)} \ud \xi\,.
\]

Denote by $\on{Diff}^1(\R^d)$ the space of $C^1$-diffeomorphisms of $\R^d$, i.e.,
\[
\on{Diff}^1(\R^d) = \{ \ph \in C^1(\R^d,\R^d) \,:\, \ph \text{ bijective, } \ph\i \in C^1(\R^d,\R^d) \}\,.
\] 
For $s > d/2+1$ and $s \in \R$ there are three equivalent ways to define the group $\mc D^s(\R^d)$ of Sobolev diffeomorphisms:
\begin{align*}
\mc D^s(\R^d) &= \{ \ph \in \on{Id} + H^s(\R^d,\R^d) \,:\, \ph \text{ bijective, }
\ph\i \in \on{Id} + H^s(\R^d,\R^d) \} \\
&= \{ \ph \in \on{Id} + H^s(\R^d,\R^d) \,:\, 
\ph \in \on{Diff}^1(\R^d) \} \\
&= \{ \ph \in \on{Id} + H^s(\R^d,\R^d) \,:\, 
\det D\ph(x) > 0,\, \forall x \in \R^d \}\,.
\end{align*}
If we denote the three sets on the right by $A_1$, $A_2$ and $A_3$, then it is not difficult to see the inclusions $A_1 \subseteq A_2 \subseteq A_3$. The equivalence $A_1 = A_2$ has first been shown in \cite[Sect. 3]{Ebin1970b} for the diffeomorphism group of a compact manifold; a proof for $\mc D^s(\R^d)$ can be found in \cite{Inci2013}. Regarding the inclusion $A_3 \subseteq A_2$, it is shown in \cite[Cor. 4.3]{Palais1959} that if $\ph \in C^1$ with $\det D\ph(x) > 0$ and $\lim_{|x |\to \infty} | \ph(x)| = \infty$, then $\ph$ is a $C^1$-diffeomorphism.

It follows from the Sobolev embedding theorem, that $\mc D^s(\R^d) - \on{Id}$ is an open subset of $H^s(\R^d,\R^d)$ and thus a Hilbert manifold. Since each $\ph \in \mc D^s(\R^d)$ has to decay to the identity for $|x|\to \infty$, it follows that $\ph$ is orientation preserving. More importantly, $\mc D^s(\R^n)$ is a topological group, but not a Lie group, since left-multiplication and inversion are continuous, but not smooth.

The space of vector fields on $\R^d$ is either $\mf X^s(\R^d)$ or $H^s(\R^d,\R^d)$ and we shall denote by $\mc D^s(\R^d)_0$ the connected component of the identity in $\mc D^s(\R^d)$.

\subsection{Boundedness of Composition}

We will use the following lemma in the later parts of the paper to estimate composition in Sobolev spaces. The first two parts are Cor. 2.1 and Lem. 2.7 of \cite{Inci2013}, the third statement is a slight refinement of \cite[Lem. 2.11]{Inci2013} and can be proven in the same way. Denote by $B_\ep(0)$ the $\ep$-ball around the origin in $H^s(\R^d,\R^d)$.

\begin{lemma}\label{CompositionsLemma}
Let $s > d/2+1$ and $0 \leq s' \leq s$.
\begin{enumerate}
\item
Given $\ps \in \mc D^s(\R^d)$ there exists $\ep > 0$ and $M>0$, such that $\ps + B_\ep(0) \subseteq \mc D^s(\R^d)$ and
\[
\inf_{x \in \R^d} \det D\ph(x) > M\quad \text{ for all } \ph \in \ps + B_\ep(0)\,.
\]
\item
\label{lem:item2_technical_composition}
Given $M, C > 0$ there exists $C_{s'} = C_{s'}(M, C)$, such that for all $\ph \in \mc D^s(\R^d)$ with
\[
\inf_{x \in \R^d} \det D\ph(x) > M\quad\text{and}\quad
\| \ph - \on{Id} \|_{H^s} < C\,,
\]
and all $f \in H^{s'}(\R^d)$,
\[
\| f \o \ph\|_{H^{s'}} \leq C_{s'} \| f \|_{H^{s'}}\,.
\]
\item
Assume additionally $s' > d/2$. Let $U \subset \mc D^s(\R^d)$  be a convex set and $M, C>0$ constants, such that
\[
\inf_{x \in \R^d} \det D\ph(x) > M\;\;\text{and}\;\;
\| \ph - \on{Id} \|_{H^s} < C\quad \text{for all } \ph \in U\,.
\]
Then there exists $C_{s'} = C_{s'}(M, C)$, such that for all $f \in H^{s'+1}(\R^d)$ and $\ph, \ps \in U$,
\[
\| f \o \ph - f \o \ps \|_{H^{s'}} \leq C_{s'} \| f \|_{H^{s'+1}} \| \ph - \ps \|_{H^{s'}}\,.
\]
\end{enumerate}
\end{lemma}

\begin{proof} For the sake of completeness we give a proof of the third statement. We may assume that $f \in C^\infty_c(\R^d)$, since $C^\infty_c(\R^d)$ is dense in $H^{s'+1}(\R^d)$. Introduce $\de \ph(x) = \ph(x) - \ps(x)$ and note that $\ph + t \de \ph \in U$ for any $0 \leq t \leq 1$. Since $\ph, \ps \in \on{Diff}^1_+(\R^d)$, we have for all $x \in \R^d$,
\begin{align*}
f \o \ph(x) - f \o \ps(x) &= \int_0^1 \frac{d}{dt}
\left( f \o (\ph+t \de \ph)(x) \right) \ud t \\
&= \int_0^1 Df\left((\ph+t \de\ph)(x)\right).\de \ph(x) \ud t\,.
\end{align*}
Hence
\begin{align*}
  \left\| f \o \ph - f \o \ps \right\|_{H^{s'}} &\leq C'_{s'}
\int_0^1 \left\| Df \o (\ph + t \de \ph) \right\|_{H^{s'}} \| \ph - \ps \|_{H^{s'}} \ud t \\
&\leq C_{s'}'' \| Df \|_{H^{s'}} \| \ph - \ps \|_{H^{s'}} \leq C_{s'} \| f \|_{H^{s'+1}} \| \ph - \ps \|_{H^{s'}}\,,
\end{align*}
with some constants $C_{s'}, C_{s'}', C_{s'}''$.
\end{proof}

\section{Convergence of Flows in $\mc D^s(\R^d)$}

In this section we want to clarify, what is meant by the flow of a vector field -- in particular for vector fields that are only $L^1$ -- and then prove some results about the convergence of flows given convergence of the underlying vector fields. The main result of the section is Thm. \ref{thm:flow_convergence}, which shows that for $s > d/2+1$ the flow map -- assuming it exists -- is continuous as a map
\[
\on{Fl} : L^1(I,H^s(\R^d,\R^d)) \to C(I, \mc D^{s'}(\R^d))\,,
\]
where $d/2+1 < s' < s$. The result will be strengthened by Thm. \ref{thm:d2plus1_flow}, which will show the existence of the flow as well as the convergence for $s'=s$.

\subsection{Pointwise and $\mc D^s$-valued flows}
\label{sec:flow_def}
Let $s>d/2+1$ and $I$ be a compact interval containing 0. Assume $u$ is a vector field, $u \in L^1(I,H^s(\R^d,\R^d))$. It is shown in~\cite[Sect. 8.2]{Younes2010} that there exists a map $\ph : I \x \R^d \to \R^d$, such that
\begin{itemize}
\item
$\ph(\cdot,x)$ is absolutely continuous for each $x$ and
\item
$\ph(t,\cdot)$ is continuous for each $t$,
\end{itemize}
and this map satisfies the equation
\begin{equation}
\label{eq:flow_pointwise}
\ph(t,x) = x + \int_0^t u(\ta, \ph(\ta, x)) \ud \ta\,.
\end{equation}
We will call such a map $\ph$ the {\it pointwise flow of $u$} or simply the {\it flow of $u$}. It then follows that for each $x \in \R^d$ the differential equation
\[
\p_t \ph(t,x) = u(t, \ph(t,x))
\]
is satisfied $t$ almost everywhere. It is also shown in \cite[Thm. 8.7]{Younes2010} that $\ph(t)$ is a $C^1$-diffeomorphism for all $t \in I$.

We will denote by $\on{Fl}(u) : I \to \on{Diff}^1(\R^d)$ the flow map of the vector field $u$. Given $t \in I$, the \emph{flow at time $t$} is $\on{Fl}_t(u) \in \on{Diff}^1(\R^d)$. If $\ph$ is the map solving~\eqref{eq:flow_pointwise}, then $\ph = \on{Fl}(u)$ and $\ph(t) = \on{Fl}_t(u)$. Note that \eqref{eq:flow_pointwise} implies $\on{Fl}_0(u) = \on{Id}$; we shall use this convention throughout the paper.

If we additionaly assume that $\ph \in C(I, \mc D^s(\R^d))$, i.e., $\ph$ is a continuous curve in $\mc D^s(\R^d)$, then Lem. \ref{lem:flow_Ds_implies_pointwise} shows that the function $t \mapsto u(t) \o \ph(t)$ is Bochner integrable in $H^s$ and the identity
\begin{equation}
\label{eq:flow_Ds}
\ph(t) = \on{Id} + \int_0^t u(\ta) \o \ph(\ta) \ud \ta
\end{equation}
holds in $\mc D^s(\R^d)$; furthermore, \eqref{eq:flow_Ds} implies that the curve $t \mapsto \ph(t)$ is absolutely continuous. We will call a curve $\ph \in C(I,\mc D^s(\R^d))$ a {\it flow of $u$ with values in $\mc D^s(\R^d)$} or a {\it $\mc D^s$-valued flow of $u$}. The pointwise flow of a vector field is unique and therefore, if the $\mc D^s$-valued flow exists, it is also unique. It will be shown in Thm. \ref{thm:d2plus1_flow} that every vector field $u \in L^1(I,H^s)$ has a $\mc D^s$-valued flow.

\begin{lemma}
\label{lem:flow_Ds_implies_pointwise} Let $s > d/2+1$, $u \in L^1(I,H^s(\R^d,\R^d))$ and $\ph \in C(I,\mc D^s(\R^d))$. Then it follows that:
\begin{enumerate}
\item
The function $t \mapsto u(t) \o \ph(t)$ is Bochner integrable.
\item
If $\ph$ satisfies \eqref{eq:flow_pointwise}, then the identity \eqref{eq:flow_Ds} holds as an identity in $\mc D^s(\R^d)$.
\end{enumerate}
\end{lemma}

\begin{proof}
First we show that $t \mapsto u(t) \o \ph(t)$ is Bochner integrable. The map $t \mapsto u(t) \o \ph(t)$ is weakly measurable and since $H^s$ is separable, also measurable \cite[Prop.~1.1.10]{Schwabik2005}. Since $I$ is compact, the set $\ph(I)$ satisfies the conditions of Lem. \ref{CompositionsLemma}~\eqref{lem:item2_technical_composition}, i.e., there exists a constant $C$ such that
\begin{equation*}
\| v \o \ph(t) \|_{H^s} \leq C \| v \|_{H^s}\,,
\end{equation*}
holds for all $v \in H^s$ and all $t \in I$. Thus
\[
\int_I \| u(t) \o \ph(t) \|_{H^s} \ud t \leq C \| u \|_{L^1} < \infty\,,
\]
via \cite[Thm.~1.4.3]{Schwabik2005}, which implies that $t \mapsto u(t) \o \ph(t)$ is Bochner integrable.

Now we prove the second statement. Denote by $\on{ev}_x : H^s(\R^d,\R^d) \to \R^d$ the evaluation map. Since $s > d/2$, this map is continuous and thus \eqref{eq:flow_pointwise} can be interpreted as
\[
\on{ev}_x\left( \ph(t) - \on{Id} \right) = \int_0^t \on{ev}_x \left( u(\ta) \o \ph(\ta) \right) \ud \ta\,.
\]
The Bochner integral commutes with bounded linear maps \cite[Thm.~6]{Diestel1977}, and the set $\{ \on{ev}_x \,:\, x \in \R^d \}$ is point-separating. Thus we obtain
\[
\ph(t) -\on{Id} = \int_0^t u(\ta) \o \ph(\ta) \ud \ta\quad\text{ in }H^s(\R^d,\R^d)\,,
\]
which concludes the proof.
\end{proof}

The meaning of Lem.~\ref{lem:flow_Ds_implies_pointwise} is that the notions of $\mc D^s$-valued flow and pointwise flow coincide, if we know a priori, that $\ph$ is a continuous curve in $\mc D^s(\R^d)$. The next lemma shows the basic property, that being a flow is preserved under uniform convergence of the flows and $L^1$-convergence of the vector fields.

\begin{lemma}
\label{lem:uniform_convergence_of_flows}
Let $s>d/2+1$ and let $u^n \in L^1(I,H^s(\R^d,\R^d))$ be a sequence of vector fields with $\mc D^s$-valued flows $\ph^n$. Assume that $u^n \to u$ and $\ph^n -\ph \to 0$ in $L^1(I,H^s)$ and $C(I, H^s)$ respectively. Then $\ph$ is the $\mc D^s$-valued flow of $u$.
\end{lemma}

\begin{proof}
We need to show two things: that $\ph(t) \in \mc D^s(\R^d)$ and that $\ph$ is the $\mc D^s$-valued flow of $u$. First note that $\ph^n(t) - \ph(t) \in H^s$ implies $\ph(t) - \on{Id} \in H^s$.

As $\ph^n$ is the flow of $u^n$, it satisfies the identity
\begin{equation}
\label{eq:ptwse_flow_again} 
\ph^n(t,x) = x + \int_0^t u^n(\ta, \ph^n(\ta, x)) \ud \ta\,,
\end{equation}
for all $(t,x) \in I \x \R^d$. From the estimates
\begin{align*}
&\left| \int_0^t u^n(\ta, \ph^n(\ta,x)) - u(\ta, \ph(\ta,x)) \ud \ta \right| \\
&\quad\leq \int_0^t \left| u^n(\ta, \ph^n(\ta,x)) - u(\ta, \ph^n(\ta,x)) \right|
+  \left| u(\ta, \ph^n(\ta,x)) - u(\ta, \ph(\ta,x)) \right| \ud \ta \\
&\quad\leq \int_0^t \left\| u^n(\ta) - u(\ta) \right\|_{\infty} 
+ \left\| Du(\ta) \right\|_{\infty} \left\| \ph^n(\ta) - \ph(\ta) \right \|_{\infty} \ud \ta \\
&\quad\leq C \int_0^t \left \| u^n(\ta) - u(\ta) \right\|_{H^s} + \| u(\ta) \|_{H^s} \| \ph^n(\ta) - \ph(\ta) \|_{H^s} \ud \ta \\
&\quad\leq C \| u^n - u \|_{L^1(I,H^s)} + C\| u \|_{L^1(I,H^s)} \| \ph^n - \ph \|_{C(I,\mc D^s)}\,,
\end{align*}
with the constant $C$ arising from Sobolev embeddings, 
we see by passing to the limit in \eqref{eq:ptwse_flow_again} that $\ph$ is the pointwise flow of $u$. As remarked at the beginning of the section, it is shown in \cite[Thm 8.7]{Younes2010} that the pointwise flow $\ph(t)$ is a $C^1$-diffeomorphism and together with $\ph(t) - \on{Id} \in H^s$ this shows $\ph(t) \in \mc D^s(\R^d)$. Finally it follows from Lem. \ref{lem:flow_Ds_implies_pointwise} that $\ph$ is the $\mc D^s$-valued flow.
\end{proof}

We will use the following decomposition method repeatedly.

\begin{remark}
\label{rem:decomposition_principle}
A recurring theme is to show the existence of the flow
\[
\on{Fl}_t : L^1(I,\mf X^s) \to \mc D^s,\qquad u \mapsto \ph(t)\,,
\]
and its continuity -- either pointwise or uniformly in $t$ -- where $\mf X^s$ is the space of vector fields of a certain Sobolev regularity $s$ on $\R^d$ or on a manifold $M$. This is often done by proving the statement in question first for small vector fields, i.e. those with $\| u \|_{L^1} < \ep$ for some given $\ep$. The statement then follows for all vector fields via the following general principle.

Let $\ep>0$ be fixed. Given a vector field $u \in L^1(I,\mf X^s)$, there exists an $N$ and a decomposition of the interval $I$ into $N$ subintervals $[t_j,t_{j+1}]$, such that on each subinterval we have
\[
\int_{t_j}^{t_{j+1}} \| u(t) \|_{H^s} \ud t < \ep\,.
\]

Note that, while the points $t_j$ will depend on $u$, their total number $N$ can be bounded by a bound depending only on $\|u\|_{L^1}$; indeed we have $N \leq \| u\|_{L^1}/\ep + 1$. To see this, assume w.l.o.g. that $I=[0,1]$ and define the function $f(t) = \int_0^t \| u(\ta) \|_{H^s} \ud \ta$. The function is non-decreasing and maps $[0,1]$ to $[0,\|u\|_{L^1}]$. Subdivide the latter interval into $N$ subintervals $[s_j,s_{j+1}]$ of length less than $\ep$ and set $t_0 = 0$ and $t_j = \sup f\i(s_j)$ for $j=1,\dots, N$.

Let $u_j = u|_{[t_j,t_{j+1}]}$ be the restriction of $u$ to the subinterval $[t_j,t_{j+1}]$. We have $\| u_j\|_{L^1} < \ep$ and we can apply the proven statement to obtain the existence of a flow, which we denote $\ph_j$; here we let $\ph_j(t_j) = \on{Id}$. Then we define for $t \in [t_j,t_{j+1}]$,
\[
\ph(t) = \ph_{j}(t) \o \ph_{j-1}(t_j) \o \dots \o \ph_1(t_2) \o \ph_0(t_1)\,.
\]
It can easily be checked, that $\ph$ is the flow of $u$ -- on $\R^d$ this can be done directly and on a manifold $M$ using coordinate charts. As the flow is put together using only finitely many compositions and $\mc D^s$ is a topological group any statement about continuity of the flow map can be transferred from $u_j$ to $u$.

Another reformulation of the decomposition principle is that any diffeomorphism $\ph$, that is the flow of a vector field $u$ with $\| u \|_{L^1} < r$, can be decomposed into
\[
\ph = \ph_1 \o \ph_2 \o \dots \o \ph_N\,,
\]
where each $\ph_j$ is the flow of a vector field $u_j$ with $\| u_j \|_{L^1} < \ep$ and $N$ depends only on $r$.
\end{remark}

A first example, that uses this method is the proof of the following lemma, showing that Lem. \ref{CompositionsLemma} can be applied on arbitrary geodesic balls.

\begin{lemma}
\label{lem:composition_bound_Rd}
Let $s>d/2+1$ and $0\leq s' \leq s$. Given $r>0$ and $n \in \mb N$, there exists a constant $C$, such that the inequality
\[
\| v \o \ph \|_{H^{s'}} \leq C \| v \|_{H^{s'}}
\]
holds for all $v \in H^{s'}(\R^d,\R^n)$ and all $\ph \in \mc D^s(\R^d)$, that can be written as $\ph = \ps(1)$, where $\ps$ is the $\mc D^s(\R^d)$-valued flow of a vector field $u$ with $\| u \|_{L^1(I,H^s)} < r$.
\end{lemma}

\begin{proof}
For the purposes of this proof we set $I=[0,1]$. Choose an $\ep > 0$ such that $\on{Id}+ B_\ep(0) \subseteq \mc D^s(M)$ with $B_\ep(0)$ being the $\ep$-ball in $H^s(\R^d,\R^d)$. Using Rem.~\ref{rem:decomposition_principle} it is enough to prove the lemma for vector fields $u$ with $C \| u \|_{L^1} < \ep$. Let $\ps$ be the $\mc D^s$-valued flow of such a vector field; the existence of $\ps$ is guaranteed by the assumptions of the lemma.
We claim that $\ps$ satisfies $\ps(t) \in \on{Id} + B_\ep(0)$. Assume the contrary and let $T$ be the smallest time, such that either $\| \ps(T) - \on{Id}\|_{H^s} = \ep$ or $T=1$. Then for $t < T$ we have the bound
\begin{equation*}
\| \ps(t) - \on{Id}\|_{H^s} \leq
\int_{0}^t \| u(\ta) \o \ps(\ta) \|_{H^s} \ud \ta
\leq C \int_I \| u(\ta) \|_{H^s} \ud \ta < \ep\,.
\end{equation*}
The curve $t \mapsto \ps(t)$ is continuous in $\mc D^s(\R^d)$ and since the last inequality doesn't depend on $t$, it remains strict even in the limit $t \to T$, thus showing $\| \ps(T) - \on{Id}\|_{H^s} < \ep$. This implies that $T = 1$ and $\ph = \ps(1) \in \on{Id} + B_\ep(0)$.

This shows that given $\ph$, we can decompose $\ph$ into
\[
\ph = \ph^1 \o \dots \o \ph^N
\]
and $\ph^k \in \on{Id} + B_\ep(0)$ for all $k = 1,\dots,N$. For each $\ph^k$ we can apply Lem.~ \ref{CompositionsLemma}~\eqref{lem:item2_technical_composition} to obtain
\[
\| u \o \ph \|_{H^{s'}} \leq C^N_1 \| u \|_{H^{s'}}\,,
\]
for some constant $C_1$. As $N$ depends on $\ph$ only via $r$, this completes the proof.
\end{proof}

\begin{remark}
With a bit more work one can show that for each $r>0$, there exist constants $M$ and $C$, such that the bounds
\[
\inf_{x \in \R^d} \det D\ph(t,x) > M\;\;\text{and}\;\;
\| \ph(t) - \on{Id} \|_{H^s} < C
\]
hold for diffeomorphisms, that are flows of vector fields with $L^1$-norm less that $r$; then it is possible to apply Lem. \ref{CompositionsLemma}~\eqref{lem:item2_technical_composition} directly.
\end{remark}

The next theorem shows that $L^1$-convergence of $H^s$-vector fields implies uniform convergence of the flows, not in $\mc D^s(\R^d)$, but in $\mc D^{s'}(\R^d)$ with $s'< s$. The proof is a generalization of the proof in \cite[Prop. B.1]{Inci2012}.

\begin{theorem}
\label{thm:flow_convergence}
Let $s>d/2+1$ and let $u^n \in L^1(I,H^s(\R^d,\R^d))$ be a sequence of vector fields with $\mc D^s$-valued flows $\ph^n$. Assume that $u^n \to u$ in $L^1(I,H^s)$. 

Then there exists a map $\ph: I\x \R^d\to \R^d$, satisfying $\ph \in C(I, \mc D^{s'}(\R^d))$ for all $s'$ with $d/2+1 < s' < s$,
\[
\ph^n \to \ph \text{ in } C(I, \mc D^{s'}(\R^d))\,,
\]
and $\ph$ is the $\mc D^{s'}$-valued flow of $u$.
\end{theorem}

\begin{proof}

Let $B_\ep^s(0)$ be the $\ep$-ball in $H^s(\R^d,\R^d)$. As $s>d/2+1$ we obtain via Lem. \ref{CompositionsLemma} an $\ep>0$ and a constant $C=C(\ep)$, such that $\on{Id} + B_{\ep}^s(0) \subseteq \mc D^{s}(\R^d)$ and the estimates
\begin{align}
\label{eq:tmp_ineq2}
\| u \o \ph - u \o \ps \|_{H^{s-1}} &\leq C \| u \|_{H^s} \| \ph - \ps \|_{H^{s-1}} \\
\label{eq:tmp_ineq3}
\| u \o \ph \|_{H^{s-1}} &\leq C \| u \|_{H^{s-1}} \\
\label{eq:tmp_ineq4}
\| u \o \ph \|_{H^{s}} &\leq C \| u \|_{H^{s}}
\end{align}
are valid for all $u \in H^s$ and all $\ph, \ps \in \on{Id}+B^s_{\ep}(0)$.

{\bf Step 1.} Reduce problem to $\on{Id} + B_\ep^s(0)$. \\
Using the decomposition method of Rem.~\ref{rem:decomposition_principle} it is enough to prove the theorem for vector fields $u$ with $C \| u \|_{L^1} < \ep$. Since $u^n \to u$ in $L^1$, we can also assume that $C \| u^n \|_{L^1} < \ep$ for all $n \in \mb N$.

As part of the proof of Lem.~\ref{lem:composition_bound_Rd} it was shown that if $u^n$ satisfies $C \| u^n \|_{L^1} < \ep$, then its flow $\ph^n$ remains in $\on{Id} + B^s_\ep(0)$. Thus we can restrict our attention to diffeomorphisms lying in an $\ep$-ball around $\on{Id}$.

{\bf Step 2.} Convergence in $H^{s-1}(\R^d,\R^d)$. \\
We show that $(\ph^n(t) - \on{Id})_{n \in \mathbb N}$ are Cauchy sequences in $H^{s-1}$, uniformly in $t$. Using \eqref{eq:tmp_ineq2} and \eqref{eq:tmp_ineq3} we can estimate
\begin{align*}
\| \ph^n(t) &- \ph^m(t) \|_{H^{s-1}} \leq{}\\
&\leq \int_0^t \| u^n \o \ph^n - u^m \o \ph^n \|_{H^{s-1}} +
\| u^m \o \ph^n - u^m \o \ph^m \|_{H^{s-1}} \ud \ta \\
&\leq C \int_0^t \| u^n - u^m \|_{H^{s-1}} + \| u^m \|_{H^{s}}
\| \ph^n - \ph^m \|_{H^{s-1}} \ud \ta\,.
\end{align*}
Via Gronwall's inequality we get for some $C_1>0$, independent of $t$,
\begin{equation}
\label{eq:phn_cauchy_s-1}
\| \ph^n(t) - \ph^m(t) \|_{H^{s-1}} \leq C_1 \int_0^t \| u^n(\ta) - u^m(\ta) \|_{H^{s-1}} \ud \ta\,.
\end{equation}
Thus there exists a continuous limit curve $\ph(t) - \on{Id} \in H^{s-1}$. 

{\bf Step 3.} Convergence in $H^{s'}(\R^d,\R^d)$ with $s-1 < s' < s$. \\
We apply the following interpolation inequality, see, e.g., \cite[Lem. B.4]{Inci2012}:
\[
\| f \|_{H^{\la \si +(1-\la)s}} \leq C_2\, \| f \|_{H^{\si}}^\la \| f \|_{H^{s}}^{1-\la}\,,
\]
The inequality is valid for $0 \leq \si \leq s$, $f \in H^s(\R^d, \R^d)$ and a constant $C_2$, independent of $f$. Choose in the above inequality $\si=s-1$ and $0 < \la \leq 1$. Then
\begin{align*}
\| \ph^n(t) &- \ph^m(t) \|_{H^{s-\la}} \leq {}\\
&{}\leq C_2 \| \ph^n(t) - \ph^m(t) \|_{H^{s-1}}^\la \| \ph^n(t) - \ph^m(t) \|_{H^s}^{1-\la} \\
&{}\leq C_2 \| \ph^n(t) - \ph^m(t) \|_{H^{s-1}}^\la \left(
\| \ph^n(t) - \on{Id} \|_{H^s} + 
\| \ph^m(t) - \on{Id} \|_{H^s} \right)^{1-\la} \\
&{}\leq C_2 \| \ph^n(t) - \ph^m(t) \|_{H^{s-1}}^\la (2\ep)^{1-\la}\,.
\end{align*}
Since $\ph^n(t) - \on{Id} \to \ph(t) - \on{Id}$ in $H^{s-1}$, uniformly in $t$, it follows that $(\ph^n(t) - \on{Id})_{n \in \mathbb N}$ is a Cauchy sequence in $H^{s'}$ for $s-1 \leq s' < s$, uniformly in $t$. As $\ph^n(t) - \on{Id}$ converges to $\ph(t) - \on{Id}$ in $H^{s-1}$, it must also converge to the same limit in $H^{s'}$. By applying Lem. \ref{lem:uniform_convergence_of_flows} we see that $\ph\in \mc D^{s'}(\R^d)$ and that it is the $\mc D^{s'}$-valued flow of $u$.
\end{proof}

\section{Existence of the flow map}

The main result of this section is the existence and continuity of the flow map 
\[
\on{Fl} : L^1(I, \mf X^s(\R^d)) \to C(I, \mc D^s(\R^d))
\]
for $s>d/2+1$, with $I$ being a compact interval containing 0. This result will be the crucial ingredient in proving that the group $\mc G_{H^s(\R^d,\R^d)}$, introduced in Sect.~\ref{sec:lddmm_framework}, coincides with the connected component of the identity of $\mc D^s(\R^d)$. We would like to make some comments about this result. 

Since the flow $\ph$ of a vector field $u$ is defined as the solution of the ODE
\begin{equation}
\label{eq:flow_again}
\begin{aligned}
\p_t \ph(t) &= u(t) \o \ph(t) \\
\ph(0) &= \on{Id}
\end{aligned}\,,
\end{equation}
the first attempt at showing the existence of $\ph$ would be to consider \eqref{eq:flow_again} as an ODE in $\mc D^s(\R^d)$ -- the latter being, up to translation by $\on{Id}$, an open subset of the Hilbert space $H^s(\R^d,\R^d)$ -- with the right hand side given by the vector field
\begin{equation}
\label{eq:vector_field_U}
U : I \x \mc D^s \to H^s\,,\quad U(t,\ph) = u(t) \o \ph\,.
\end{equation}
This runs into two sets of difficulties.

Firstly, the Picard--Lindelöf theory of ODEs requires the right hand side $f(t,x)$ of an ODE to be (locally) Lipschitz continuous in $x$ and continuous in $t$. Under these conditions the theorem of Picard-Lindelöf guarantees the local existence of integral curves. In our case the right hand side is not continuous in $t$, but only $L^1$. The usual way to prove existence of solutions in the framework of Picard--Linderlöf involves the Banach fixed point theorem, and the proof can be generalized without much difficulty to ODEs, that are not continuous in $t$. It is enough to require that $f(t,x)$ is Lipschitz in $x$ and only measurable in $t$ and that the Lipschitz constants are locally integrable, i.e., there exists a function $\ell(t)$ with $\int \ell(t) \ud t < \infty$, such that
\[
\| f(t,x_1) - f(t,x_2)\| \leq \ell(t)\, \| x_1 - x_2 \|
\]
is valid for all $x_1,x_2$ and for $t$ almost everywhere. This class of differential equations is called ordinary differential equations of \emph{Carathéodory type}. We have summarized the key facts about ODEs of Carathéodory type in App. \ref{sec:caratheodory}.

Secondly, the vector field $U$ from \eqref{eq:vector_field_U} is also not Lipschitz in $\ph$. The composition map $H^s \x \mc D^s \to H^s$ is continuous, but not Lipschitz continuous. In finite dimensions the theorem of Peano shows that vector fields $f(t,x)$ that are continuous in $t$ and $x$, have flows, but the flows might fail to be unique. In infinite dimensions this is not the case anymore; an example of a continuous vector field without a flow can be found in \cite[Example 2.1]{Deimling1977}.

For a continuous vector field $u$, i.e., $u \in C(I, H^s)$, the existence of a $\mc D^s$-valued flow has been shown in \cite{Fischer1972} and using different methods also in \cite{Bourguignon1974} and \cite{Inci2012}. We will briefly review the proofs to choose the one, that most easily generalizes to vector fields $u \in L^1(I,H^s)$.

If we only require $s>d/2+2$, then the proof is much shorter than the more general case $s>d/2+1$ and can be found already in \cite{Ebin1970}. First one considers the equation \eqref{eq:flow_again} as an ODE on $\mc D^{s-1}(\R^d)$. Due to the properties of the composition map, the vector field $U : I \x \mc D^{s-1} \to H^{s-1}$ is a $C^1$-vector field and hence has a $\mc D^{s-1}$-valued flow $\ph$. This is worked out in detail in Lem. \ref{lem:d2plus2_flow}. To show that $\ph \in \mc D^s$, one considers the differential equation for $D\ph(t)$,
\[
\p_t \left(D\ph(t) - \on{Id}_{d\x d}\right) = 
\left(Du(t) \o \ph(t)\right).\left(D\ph(t) - \on{Id}_{d\x d}\right)
 + Du(t) \o \ph(t)\,.
\]
This is a linear differential equation on $H^{s-1}$, thus showing $D\ph - \on{Id}_{d\x d} \in H^{s-1}$ and $\ph \in \mc D^s$. The details of this argument can be found in Lem. \ref{lem:flow_regularity}.

Improving the hypothesis on $s$ to $s>d/2+1$ requires a bit of work. For vector fields $u \in C(I, H^s)$ that are continuous in time and not just $L^1$ this result has been proven by three different methods.

\begin{asparaenum}
\item
The approach used in \cite{Fischer1972} was to derive an equation for $\ph\i(t)$ instead of $\ph(t)$. Write $\ph\i(t) = \on{Id} + f(t)$ with $f(t) \in H^s$. Then $\p_t \ph\i(t) = -D\ph\i(t).u(t)$ and so $f(t)$ satisfies the equation
\begin{equation}
\label{eq:hyperbolic_flow}
\p_t f(t) = -Df(t).u(t) - u(t)\,.
\end{equation}
This is a linear, symmetric, hyperbolic system and the theory developed in \cite{Fischer1972} can be applied to show that, given $u \in C(I,H^s)$, the system \eqref{eq:hyperbolic_flow} has a solution $f(t) \in H^s$ and hence $\ph\i(t) \in \mc D^s(\R^d)$. To extend this method to vector fields that are only $L^1$ in $t$, one would need a theory of linear, hyperbolic systems with non-smooth (in $t$) coefficients.

\item
The method of \cite{Bourguignon1974} considers not only the groups $\mc D^s(\R^d)$ which are based on the spaces $H^s$, but the more general family $W^{s,p}$ and the corresponding diffeomorphism groups, which we shall denote by $\mc D^{s,p}(\R^d)$. One proves that vector fields $u \in C(I, W^{s,p})$ with $s>d/p+1$ have $\mc D^{s,p}$-valued flows. The proof considers only $s \in \mathbb N$ and proceeds by induction on $s$. The induction step uses the fact that given $s$ satisfying $s > d/p+1$ we can find $p'> p$ such that $s-1 > d/p'+1$ and hence we can apply the induction hypothesis to the pair $(s-1,p')$. Extending this method to $s \in \R$ and vector fields $u \in L^1(I, W^{s,p})$ would require us to study properties of the composition map on the spaces $\mc D^{s,p}(\R^d)$ -- this has not yet been done for $s \in \R \setminus \mathbb N$.

\item
The idea of \cite[App. B]{Inci2012} is to approximate a vector field $u \in C(I,H^s)$ by a sequence of vector fields in $H^{s+1}$ and then to show that the corresponding flows converge as well. This method is ideally suited to be generalised from continuous vector fields to $L^1$ vector fields and it will be the path we choose to follow here.
\end{asparaenum}

To prepare the proof of the main theorem, Thm. \ref{thm:d2plus1_flow}, we will need some lemmas. The first lemma -- which can be traced back to \cite[Lem. 3.3]{Ebin1970} -- shows that the flow of a vector field is as regular as the vector field itself.

\begin{lemma}
\label{lem:flow_regularity}
Let $d/2+1 < s' \leq s$ and $u \in L^1(I, H^s(\R^d,\R^d))$. Assume $u$ has a flow in $\mc D^{s'}(\R^d)$. Then in fact $\ph \in C(I, \mc D^{s}(\R^d))$.
\end{lemma}

\begin{proof}
We will first prove the case $s' < s \leq s'+1$. This is equivalent to $s - 1 \leq s' < s$. Our aim is to show that $D\ph(t) - \on{Id}_{d\x d}$ is a continuous curve in $H^{s-1}(\R^d,\R^{d\x d})$, implying that $\ph(t) - \on{Id}$ is a continuous curve in $H^s(\R^d, \R^d)$. 
Note that the derivative $D\ph(t)$ satisfies the following ODE in $H^{s'-1}$, $t$-a.e.,
\begin{equation}
\label{eq:Dphi_ODE}
\p_t \left(D\ph(t) - \on{Id}_{d\x d}\right) = 
(Du(t) \o \ph(t)).\left(D\ph(t) - \on{Id}_{d\x d}\right) + Du(t) \o \ph(t)\,.
\end{equation}
Consider the following linear, inhomogeneous, matrix-valued differential equation
\begin{equation}
\label{eq:Dphi_general_ODE}
\p_t A(t) = (Du(t) \o \ph(t)). A(t) + Du(t) \o \ph(t)\,,
\end{equation}
on $H^{s-1}(\R^d,\R^{d\x d})$. Since $H^{s-1}$ is a Banach algebra, we can interpret $Du(t) \o \ph(t)$ as an element of $L(H^{s-1})$, i.e., a linear map from $H^{s-1}$ to itself, and there exists a constant $C>0$, such that
\[
\| Du(t) \o \ph(t) \|_{L(H^{s-1})} \leq C \| Du(t) \o \ph(t) \|_{H^{s-1}}\,.
\]
Lemma \ref{lem:flow_Ds_implies_pointwise} shows that $Du(t)\o\ph(t)$ is Bochner integrable in $H^{s'}$ and thus in $H^{s-1}$. This allows us to apply the existence theorem for linear Carathéodory equations, Thm. \ref{thm:caratheodory_linear}, giving us a solution $A \in C(I, H^{s-1})$ of \eqref{eq:Dphi_general_ODE}. Since $D\ph - \on{Id}_{d \x d}$ satisfies \eqref{eq:Dphi_ODE} in $H^{s'-1}$ and $A(t)$ satisfies \eqref{eq:Dphi_general_ODE} in $H^{s-1}$, it follows that they are equal, $D\ph(t) - \on{Id}_{d\x d} = A(t)$, thus showing that $D\ph(t) - \on{Id}_{d\x d} \in H^{s-1}$.

In the general case we have $s' + k < s \leq s' + k + 1$ with $k \in \mathbb N$. The argument above proved the lemma for $k=0$. If $k \geq 1$, we apply the above argument with the pair $(s', s'+1)$ in the place of $(s', s)$. This shows that $\ph(t) \in \mc D^{s'+1}$. Then we can apply the argument with $(s'+1, s'+2)$ to obtain $\ph(t) \in \mc D^{s'+2}$ and so one shows inductively
\[
\ph(t) \in \mc D^{s'}
\Rightarrow
\ph(t) \in \mc D^{s'+1}
\Rightarrow \dots \Rightarrow
\ph(t) \in \mc D^{s'+k}
\Rightarrow
\ph(t) \in \mc D^{s}\,.
\]
In the last step we use the argument with the pair $(s'+k, s)$ to conclude that $\ph(t) \in \mc D^s$.
\end{proof}

As stated in the introduction to this section, we will first show the existence of flows for $H^s$ vector fields, when $s>d/2+2$. This involves applying the existence theorem for Carathéodory differential equations to the equation \eqref{eq:flow_again}.

\begin{lemma}
\label{lem:d2plus2_flow}
Let $s>d/2+2$ and $u \in L^1([0,1], H^s(\R^d,\R^d))$. Then $u$ has a flow in $\mc D^{s}(\R^d)$.
\end{lemma}

\begin{proof}
Define for $\ep>0$ the open ball
\[
B^{s-1}_\ep(0) = \left\{ f \in H^{s-1}(\R^d,\R^d) \,:\, \| f\|_{H^{s-1}} < \ep \right\}\,.
\]
Since $s-1>d/2+1$, we obtain by Lem. \ref{CompositionsLemma}
an $\ep>0$ and a constant $C=C(\ep)$, such that $\on{Id} + B^{s-1}_\ep(0) \subseteq \mc D^{s-1}(\R^d)$ and the estimates
\begin{align*}
\| u \o \ph_1 - u \o \ph_2 \|_{H^{s-1}} &\leq C \| u \|_{H^s} \| \ph_1 - \ph_2 \|_{H^{s-1}} \\
\| u \o \ph \|_{H^{s-1}} &\leq C \| u \|_{H^{s-1}}
\end{align*}
are valid for all $u \in H^{s}$ and all $\ph, \ph_1, \ph_2 \in \on{Id}+B^{s-1}_\ep(0)$.

Using the decomposition method in Rem. \ref{rem:decomposition_principle} it is enough to show the existence of the flow when $C \| u \|_{L^1} < \ep$. Under this assumption, define the vector field
\[
U : I \x B^{s-1}_{\ep}(0) \to H^{s-1}(\R^d,\R^d)\,,\qquad U(t,f) = u(t) \o (\on{Id}+f)\,,
\]
where $u(t)$ is given. The mapping $U$ has the Carathéodory property, Def. \ref{def:caratheodory}, because composition is continuous in $\mc D^{s-1}(\R^d)$ and $H^{s-1}$ is separable. The functions $m(t)$ and $\ell(t)$ required in Thm. \ref{thm:caratheodory_exist} are given by $m(t) = C\,\|u(t)\|_{H^{s-1}}$ and $\ell(t) = C\,\|u(t)\|_{H^s}$. 
Then by Thm. \ref{thm:caratheodory_exist} we have a solution $\ph \in C([0, 1], \mc D^{s-1}(\R^d))$ of the equation
\[
\ph(t) = \on{Id} + \int_{0}^t u(\ta) \o \ph(\ta) \ud \ta\,.
\]
Thus $\ph$ is the $\mc D^{s-1}(\R^d)$-valued flow of $u$ and Lem. \ref{lem:flow_regularity} shows that in fact $\ph$ is $\mc D^{s}(\R^d)$-valued.
\end{proof}

The next lemma shows how to approximate vector fields in $H^s(\R^d)$ by a sequence of vector fields in $H^{s+1}(\R^d)$, whilst preserving integrability in time.

\begin{lemma}
\label{lem:uk_approx}
Let $s\geq 0$ and $f \in L^1(I,H^s(\R^d))$. For $k\geq 0$, define $\ch(\xi) = \mathbbm 1_{\{|\xi| \leq k\}}(\xi)$ and let $\ch_k(D)$ be the corresponding Fourier multiplier. Then
\[
\ch_k(D) f \in L^1\!\left(I, H^{s+1}(\R^d)\right)\,,
\]
and $\ch_k(D) f \to f$ for $k \to \infty$ in $L^1(I, H^s(\R^d))$.
\end{lemma}

\begin{proof}
We have for all $t \in I$,
\[
\| \ch_k(D) f(t)\|_{H^{s+1}(\R^d)}^2 =
\int_{|\xi|\leq k} (1+|\xi|^2)^{s+1} |\wh {f(t)}(\xi)|^2 \ud \xi \leq
(1+k^2) \| f(t)\|_{H^{s}(\R^d)}^2\,,
\]
and thus $\ch_k(D) f \in L^1\!\left(I, H^{s+1}(\R^d)\right)$; in fact we have $\ch_k(D) f(t) \in H^\infty$, but this will not be needed here.

To show convergence we note that
\[
\| \ch_k(D) f(t) - f(t) \|_{H^{s}(\R^d)}^2 = 
\int_{|\xi|>k} (1+|\xi|^2)^s |\wh {f(t)}(\xi)|^2 \ud \xi \leq
\| f(t)\|_{H^{s}(\R^d)}^2\,.
\]
By the theorem of dominated convergence we obtain first
\[
\int_{|\xi|>k} (1+|\xi|^2)^s |\wh {f(t)}(\xi)|^2 \ud \xi \to 0\,,
\]
for all $t \in I$ and thus $\ch_k(D) f(t) \to f(t)$ in $H^s(\R^d)$, and by applying it again
\[
\lim_{k\to\infty} \| \ch_k(D) f - f\|_{L^1(I,H^s)} = 
\int_0^1 \lim_{k\to \infty} \| \ch_k(D) f(t) - f(t) \|_{H^{s}(\R^d)} \ud t = 0
\]
showing that $\ch_k(D) f \to f$ in $L^1$.
\end{proof}

We are now ready to prove the main theorem.

\begin{theorem}
\label{thm:d2plus1_flow}
Let $s>d/2+1$ and $u \in L^1(I, H^s(\R^d,\R^d))$. Then $u$ has a $\mc D^{s}(\R^d)$-valued flow and the map
\[
\on{Fl} : L^1(I, H^s(\R^d,\R^d)) \to C(I, \mc D^{s}(\R^d))\,,\quad u \mapsto \ph
\]
is continuous.
\end{theorem}

\begin{proof}
Given $u \in L^1(I, H^s)$, it follows from Lem. \ref{lem:uk_approx} that there exists a sequence $u^n \in L^1(I, H^{s+1})$ converging to $u$,
\[
u^n \to u \text{ in } L^1(I,H^s(\R^d,\R^d))\,.
\]
According to Lem. \ref{lem:d2plus2_flow}, each $u^n$ has a $\mc D^s(\R^d)$-valued flow; in fact they have $\mc D^{s+1}(\R^d)$-valued flows. As $u^n \to u$ in $L^1$, it was shown in Thm. \ref{thm:flow_convergence} that $u$ itself has a $\mc D^{s'}(\R^d)$-valued flow $\ph$ for each $s'$ with $d/2+1 < s'< s$ and that $\ph^n \to \ph$ in $C(I, \mc D^{s'}(\R^d))$. Finally we use the regularity result from Lem. \ref{lem:flow_regularity} to conclude that the flow $\ph$ of $u$ is $\mc D^s(\R^d)$-valued.

To prove the continuity of the flow map, consider a sequence $u^n$ converging to $u$ in $L^1(I, H^s)$ and denote by $\ph^n$ and $\ph$ the $\mc D^s$-valued flows of $u^n$ and $u$ respectively. The $H^s$-norm $\| u\|_{H^s}$ is equivalent to the norm $\| u \|_{L^2} + \| D u\|_{H^{s-1}}$ and since $\ph^n(t) \to \ph(t)$ uniformly in $\mc D^{s-1}(\R^d)$, we only need to show that $D\ph^n(t) - D\ph(t) \to 0$ uniformly in $H^{s-1}$. We will do this by applying Gronwall's lemma to
\[
D\ph^n(t) - D\ph(t) = \int_{0}^t \left( Du^n(\ta) \o \ph^n(\ta) \right). D\ph^n(\ta) - 
\left( Du(\ta) \o \ph(\ta) \right). D\ph(\ta) \ud \ta\,.
\]
Taking norms we obtain
\begin{align*}
\| D\ph^n(t) &- D\ph(t) \|_{H^{s-1}} 
\leq{} \\
&\leq \int_{0}^t \left\| \left( Du^n(\ta) \o \ph^n(\ta) \right). 
\left(D\ph^n(\ta) - D \ph(\ta)\right) \right\|_{H^{s-1}} + \\
&\qquad\qquad+\left\| \left( Du^n(\ta) \o \ph^n(\ta) - 
Du(\ta) \o \ph(\ta) \right). D\ph(\ta) \right\|_{H^{s-1}} \ud \ta\\
&\leq \int_{0}^t C \left\| Du^n(\ta) \o \ph^n(\ta) \right\|_{H^{s-1}} 
\left\| D\ph^n(\ta) - D \ph(\ta) \right\|_{H^{s-1}} + \\
&\;\;\quad{}+\left\| Du^n(\ta) \o \ph^n(\ta) - 
Du(\ta) \o \ph(\ta)  \right\|_{H^{s-1}} \cdot {}\\
&\qquad\qquad\qquad\qquad\qquad\qquad {}\cdot \left(1 + C
 \left\| D\ph(\ta) - \on{Id}_{d\x d}\right\|_{H^{s-1}} \right) \ud \ta
\end{align*}
and the constant $C$ arises from the boundedness of pointwise multiplication.

Choose $s'$ with $s-1 < s' < s$ and $s'>d/2+1$. As $\ph(I) \subset \mc D^{s'}(\R^d)$ is compact and $\ph^n(t) \to \ph(t)$ uniformly in $\mc D^{s'}(\R^d)$, it follows that the set $\{ \ph^n(t) : t \in I, n \in \mathbb N\}$ satisfies the assumptions of Lem. \ref{CompositionsLemma} (\ref{lem:item2_technical_composition})., i.e., $\det D\ph^n(t,x)$ is bounded from below and $\|\ph^n(t) - \on{Id}\|_{H^{s'}}$ is bounded from above. Thus
\[
\| D u^n(\ta) \o \ph^n(\ta) \|_{H^{s-1}} \leq C_1 \| Du^n(\ta)\|_{H^{s-1}} \leq C_2 \|u^n(\ta)\|_{H^s}\,.
\]
Also note that $\| D\ph(\ta) -\on{Id}_{d\x d}\|_{H^{s-1}}$ is bounded, since $\ph(I)$ is compact in $\mc D^s(\R^d)$. Next we estimate -- omitting the argument $\ta$ from now on --
\begin{align*}
\left\| Du^n \o \ph^n - 
Du \o \ph \right\|_{H^{s-1}}
&\leq \left\| \left( Du^n - Du \right) \o \ph^n \right\|_{H^{s-1}} + 
\left\| Du \o \ph^n - Du \o \ph  \right\|_{H^{s-1}} \\
& \leq C_2 \| u^n - u \|_{H^s} + \left\| Du \o \ph^n - Du \o \ph  \right\|_{H^{s-1}}\,.
\end{align*}
Hence
\begin{align*}
\| D\ph^n(t) &- D\ph(t) \|_{H^{s-1}}
\leq C_3 \int_{0}^t \left\| u^n \right\|_{H^s}
\left\| D\ph^n - D \ph \right\|_{H^{s-1}} \ud \ta +{} \\
&\quad + C_4 \| u^n - u \|_{L^1(I,H^s)} + C_5 \int_0^1 \left\| Du \o \ph^n - Du \o \ph  \right\|_{H^{s-1}} \ud \ta\,.
\end{align*}
In the last integral we note that since composition is a continuous map $H^{s-1} \x \mc D^{s'} \to H^{s-1}$, the integrand converges pointwise to 0 as $n\to \infty$. Because
\[
\| Du \o \ph^n - Du \o \ph\|_{H^{s-1}} \leq 2C_1 \| Du\|_{H^{s-1}} \leq 2C_1 \| u \|_{H^s}\,,
\]
we can apply the theorem of dominated convergence to conclude that
\[
\int_0^1 \left\| Du \o \ph^n - Du \o \ph  \right\|_{H^{s-1}} \ud \ta \to 0 \text{ as } n \to \infty\,.
\]
Thus we obtain via Gronwall's inequality
\begin{multline*}
\| D\ph^n(t) - D\ph(t) \|_{H^{s-1}} \leq \\
\leq \left( C_4 \| u^n - u \|_{L^1(I,H^s)} 
+ C_5 \int_0^1 \left\| Du \o \ph^n - Du \o \ph  \right\|_{H^{s-1}} \ud \ta \right) \cdot \\
\cdot \left(1 + C_3 \| u^n \|_{L^1(I,H^s)} \exp\left(\| u^n \|_{L^1(I,H^s)}\right)\right)\,,
\end{multline*}
the required uniform convergence of $D\ph^n(t) -D\ph(t) \to 0$ in $H^{s-1}$.
\end{proof}

\section{Diffeomorphisms of a compact manifold}

\subsection{Sobolev spaces on domains} Let $U \subset \R^d$ be a Lipschitz domain, i.e., a bounded open set with a Lipschitz boundary. For $s \in \R$ we can define the Sobolev space on $U$ as the set of restrictions of functions on the whole space,
\[
H^s(U, \R^n) = \left\{ g|_U \,:\, g \in H^s(\R^d,\R^n) \right\}\,,
\]
and a norm is given by
\[
\| f \|_{H^s(U)} = \inf \left\{ \| g \|_{H^s(\R^d)} \,:\, g|_U = f \right\}\,.
\]
For each Lipschitz domain $U$ and each $s \in \R$, there exists an extension operator -- see \cite{Rychkov1999} -- i.e., a bounded linear map
\[
E_U : H^s(U, \R^n) \to H^s(\R^d,\R^n)\,.
\]

\subsection{Sobolev spaces on compact manifolds}
Throughout this section, we make the following assumption:
\begin{quote}
$M$ is a $d$-dimensional compact manifold and $N$ an $n$-di\-men\-sio\-nal manifold, both without boundary.
\end{quote}

For $s \geq 0$ a function $f : M \to \R$ belongs to $H^s(M)$, if around each point there exists a chart $\ch : \mc U \to U \subset \R^d$, such that $f \o \ch\i \in H^s(U,\R)$. Similarly the space $\mf X^s(M)$ of vector fields consists of sections $u : M \to TM$, such that around each point there exists a chart with $T\ch \o u \o \ch\i \in H^s(U,\R^d)$.

To define the spaces $H^s(M,N)$ we require $s > d/2$. A continuous map $f:M\to N$ belongs to $H^s(M,N)$, if for each point $x \in M$, there exists a chart $\ch : \mc U \to U\subseteq \R^d$ of $M$ around $x$ and a chart $\et : \mc V \to V\subseteq \R^n$ of $N$ around $f(x)$, such that $\et\o f \o\ch\i \in H^s(U,\R^n)$. If $N=\R$, then $H^s(M) = H^s(M,\R)$ and $\mf X^s(M) \subset H^s(M,TM)$ consists of those $u \in H^s(M,TM)$ with $\pi_{TM} \o u = \on{Id}_M$.

In order to define norms on $H^s(M)$ and $\mf X^s(M)$ and to introduce a differentiable structure on $H^s(M,N)$, we define, following \cite{Inci2013}, a special class of atlases.

\begin{definition}
A cover $\mc U_I = (\mc U_i)_{i \in I}$ of $M$ by coordinate charts $\ch_i : \mc U_i \to U_i \subset \R^d$ is called a \emph{fine cover}, if
\begin{enumerate}[(C1)]
\item $I$ is finite and $U_i$ are bounded Lipschitz domains in $\R^d$.
\item
If $\mc U_i \cap \mc U_j \neq \emptyset$, then 
$\ch_j\o\ch_i\i \in C^\infty_b\big( \overline{ \ch_i(\mc U_i \cap \mc U_j)}, \R^d \big)$.
\item
If $\mc U_i \cap \mc U_j \neq \emptyset$, then the boundary of $\ch_i(\mc U_i \cap \mc U_j)$ is a bounded Lipschitz domain.
\end{enumerate}
\end{definition}

The spaces $H^s(M)$ and $\mf X^s(M)$ are Hilbert spaces and a norm can be defined by choosing a fine cover $\mc U_I$ of $M$. On $H^s(M)$ the norm is
\[
\| u \|^2_{H^s, \mc U_I} = \sum_{i \in I} \left \| u \o \ch_i\i \right\|^2_{H^s(U_i)}\,.
\]
Similarly for vector fields $u \in \mf X^s(M)$ we define
\[
\| u \|^2_{H^s,\mc U_I} = \sum_{i \in I}
\left\| T\ch_i \o u \o \ch_i\i \right\|^2_{H^s(U_i,\R^d)}\,.
\]
In the above formula we identify the coordinate expression $T\ch_i \o u \o \ch_i\i : U_i \to TU_i$ with a map $U_i \to \R^d$, obtained by projecting $TU_i = U_i \x \R^d$ to the second component. The norms depend on the chosen cover, but choosing another fine cover will lead to equivalent norms. We will write $\| u\|_{H^s}$ for the norms on $H^s(M)$ and $\mf X^s(M)$.

\subsection{Diffeomorphism groups on compact manifolds}

To define a differentiable structure on $H^s(M,N)$ we introduce the notion of adapted fine covers. For details on these constructions and full proofs we refer the reader to \cite[Sect.~3]{Inci2013}. 

\begin{definition}
A triple $(\mc U_I, \mc V_I, f)$ consisting of $f \in H^s(M,N)$, a fine cover $\mc U_I$ of $M$ and a fine cover of $\mc V_I$ of $\bigcup_{i \in I} \mc V_i \subseteq N$ is called a \emph{fine cover with respect to $f$} or \emph{adapted to $f$}, if $\overline{f(\mc U_i)} \subseteq \mc V_i$ for all $i \in I$.
\end{definition}

Given $f \in H^s(M,N)$ one can show that there always exists a fine cover adapted to it. Let $(\mc U_I, \mc V_I, f)$ be such a fine cover and define the subset $\mc O^s = \mc O^s(\mc U_I, \mc V_I)$,
\[
\mc O^s = \left\{
h \in H^s(M,N) \,:\, \overline{h(\mc U_i)} \subseteq \mc V_i \right\}\,,
\]
as well as the map
\[
\imath = \imath_{\mc U_I, \mc V_I} : \mc O^s \to \bigoplus_{i \in I} H^s(U_i, \R^d),\qquad
h \mapsto \left( \et_i \o h \o \ch_i\i\right)_{i \in I}\,,
\]
where $\ch_i : \mc U_i \to U_i$ and $\et_i: \mc V_i \to V_i$ are the charts associated to $\mc U_i$ and $\mc V_i$ respectively. Then $\imath(\mc O^s)$ is a $C^\infty$-submanifold of $\bigoplus_{i \in I} H^s(U_i,\R^d)$. We define a topology on $H^s(M,N)$ by letting the sets $\mc O^s(\mc U_I, \mc V_I)$ form a basis of open sets and we use the maps $\imath_{\mc U_I,\mc V_I}$ to define a differentiable structure making $H^s(M,N)$ into a $C^\infty$-Hilbert manifold. This differentiable structure is compatible with the one introduced in \cite{Eells1966, Palais1968} and used in \cite{Ebin1970}.

For $s > d/2+1$ the diffeomorphism group $\mc D^s(M)$ can be defined by
\begin{align*}
\mc D^s(M) &= \{ \ph \in H^s(M, M) \,:\, \ph \text{ bijective, }
\ph\i \in H^s(M, M) \} \\
&= \{ \ph \in H^s(M, M) \,:\, 
\ph \in \on{Diff}^1(M) \}\,, 
\end{align*}
with $\on{Diff}^1(M)$ denoting $C^1$-diffeomorphisms of $M$. The diffeomorphism group is an open subset of $H^s(M,M)$ and a topological group.

It will later be convenient to work with fine covers $(\mc U_I, \mc V_I, \on{Id})$ of $M$ adapted to the identity map with the additional constraint, that the coordinate charts of $\mc U_I$ and $\mc V_I$ are the same, i.e., $\ch_i=\et_i|_{\mc U_i}$. Such covers can always be constructed by starting with a fine cover $\mc V_I$ of $M$ and shrinking each set $\mc V_i$ slightly to $\mc U_i$, so that the smaller sets still cover $M$ and $\ol{\mc U_i} \subseteq \mc V_i$. Then $(\mc U_I, \mc V_I, \on{Id})$ is an adapted cover.

\subsection{Flows on compact manifolds}
Given a vector field $u \in L^1(I,\mf X^s(M))$ with $I$ a compact interval containing $0$, we call a map $\ph : I\x M \to M$ the \emph{pointwise flow} of $u$, if $\ph(0,x) = x$ and for each pair $(t,x) \in I \x M$ there exists a coordinate chart $\ch: \mc U \to U$ around $x$, a chart $\et: \mc V \to V$ around $\ph(t,x)$, such that with $v = T\et \o u \o \et\i$ and $\ps = \et \o \ph \o \ch\i$ the flow equation
\[
\ps(s, y) = \ps(t, x) + \int_t^s v(\ta, \ps(\ta, y)) \ud \ta
\]
holds for $(s,y)$ close to $(t, \ch(x))$. For smooth vector fields this coincides with the usual definition of a flow.

If additionally $\ph \in C(I, \mc D^s(M))$, i.e., $\ph$ is a continuous curve with values in $\mc D^s(M)$, then we call $\ph$ the \emph{$\mc D^s(M)$-valued flow} of $u$. In this case let $(\mc U_I, \mc V_I, \ph(t))$ be a fine cover adapted to $\ph(t)$ with $t \in I$ and set $u_i(t) = T\et_i \o u(t) \o \et_i\i$ and $\ph_i(t) = \et_i \o \ph(t) \o \ch_i\i$. Then
\[
\ph_i(s) = \ph_i(t) + \int_t^s u_i(\ta) \o \ph_i(\ta) \ud \ta
\]
holds for $s$ close to $t$ as an identity in $H^s(U_i,\R^d)$.

\subsection{Existence of flows}
\label{sec:flows_in_charts}

To deal with vector fields and flows on $M$, we need to pass to coordinate charts. The following is a general technique, that will be useful throughout the section. Fix a fine cover $(\mc U_J, \mc V_J, \on{Id})$ of $M$ with respect to $\on{Id}$ with $\ch_j = \et_j|_{\mc U_j}$ and let $u \in L^1(I, \mf X^s(M))$ be a vector field. We define its coordinate expression
\[
v_j = T \et_j \o u \o \et_j\i \text{ and } v_j \in L^1(I, \mf X^s(V_j))\,,
\]
and extend these vector fields to all of $\R^d$ using the extension operators $E_{V_j}$,
\[
w_j = E_{V_j}v_j \text{ and } w_j \in L^1(I, \mf X^s(\R^d))\,.
\]
Note that the norms
\begin{equation}
\label{eq:norms_equivalent_M}
\| u \|_{L^1(I, \mf X^s(M))} \sim
\sum_{j \in J} \| v_j \|_{L^1(I, \mf X^s(V_j))} \sim
\sum_{j \in J} \| w_j \|_{L^1(I, \mf X^s(\R^d))}
\end{equation}
are all equivalent. From Thm. \ref{thm:d2plus1_flow} we know, that the vector fields $w_j$ have flows
\[
\ps_j = \on{Fl}(w_j) \text{ and } \ps_j \in C(I, \mc D^s(\R^d))\,.
\]
To glue them together to a flow of $u$, the flows $\ps_j$ must not be too far away from the identity. To ensure this, we fix $\ep$ given in Lem. \ref{lem:localize_flow} and assume from now onwards, that $\| u \|_{L^1(I, \mf X^s(M))} < \ep$. Then Lem. \ref{lem:localize_flow} implies that $\ol{\ps_j(U_j)} \subseteq V_j$ and we define
\begin{equation}
\label{eq:def_phi_on_M}
\ph(t)|_{\mc U_j} = \ch_j\i \o \ps_j(t) \o \ch_j\,.
\end{equation}
It is shown in Lem. \ref{lem:phi_well_defined}, that $\ph(t)$ is well-defined and that $\ph(t) \in \mc D^s(M)$. It also follows from \eqref{eq:def_phi_on_M} that $\ol{\ph(t)(\mc U_j)} \subseteq \mc V_j$ and thus $\ph(t) \in \mc O^s(\mc U_J, \mc V_J)$ and
\[
\imath(\ph(t)) = \left( \ps_j(t)|_{\mc U_j}\right)_{j \in J} \in
\bigoplus_{j \in J} H^s(U_j, \R^d)\,.
\]
Obviously $\ph$ is the $\mc D^s$-valued flow of $u$. This leads us to the following result on existence and continuity of the flow map.

\begin{theorem}
\label{thm:d2plus1_flow_M}
Let $s > d/2+1$ and $u \in L^1(I, \mf X^s(M))$. Then $u$ has a $\mc D^s$-valued flow $\ph$ and for each $t \in I$ the map
\[
\on{Fl}_t : L^1(I,\mf X^s(M)) \to \mc D^s(M),\qquad u \mapsto \ph(t)
\]
is continuous.

\end{theorem}

\begin{proof}
The above discussion shows the existence of a $\mc D^s$-valued flow $\ph$ for vector fields $u$ with $\| u \|_{L^1} < \ep$, with $\ep$ given by Lem. \ref{lem:localize_flow}. To show that $\on{Fl}_t$ is continuous, let $u^n \to u$ in $L^1(I,\mf X^s(M))$. Since the norms in \eqref{eq:norms_equivalent_M} are equivalent, it follows that $w^n_j \to w_j$ in $L^1(I,\mf X^s(\R^d))$ and by Thm. \ref{thm:d2plus1_flow} also $\ps^n_j \to \ps_j$ in $C(I, \mc D^s(\R^d))$. Thus we see that $\imath(\ph^n(t)) \to \imath(\ph(t))$ in $\bigoplus_{j \in J} H^s(U^j,\R^d)$, which implies $\ph^n(t) \to \ph(t)$ in $\mc D^s(M)$.

Using Rem. \ref{rem:decomposition_principle} we can extend these results from vector fields $u$ with $\| u \|_{L^1} < \ep$ to all vector fields.
\end{proof}

Now we prove the two lemmas, that were used in the discussion in \ref{sec:flows_in_charts}.

\begin{lemma}
\label{lem:localize_flow}
Let $s > d/2+1$ and $(\mc U_J, \mc V_J, \on{Id})$ be a fine cover of $M$ with respect to $\on{Id}$ with $\ch_j = \et_j$. Then there exists an $\ep > 0$, such that if $\| u \|_{L^1(I, \mf X^s(M))} < \ep$, then $\ol{\ps_j(t)(U_j)} \subseteq V_j$ for all $j \in J$.
\end{lemma}

\begin{proof}
As $(\mc U_J, \mc V_J, \on{Id})$ is a fine cover, it follows that for $U_j = \ch_j(\mc U_j)$ and $V_j = \ch_j(\mc V_j)$ we have $\ol{U_j} \subseteq V_j$ and all sets are bounded. Thus there exists $\de > 0$, such that
\[
\ol{U_j + B_\de(0)} \subseteq V_j\,,
\]
and $B_\de(0)$ is the $\de$-ball in $\R^d$. By Thm. \ref{thm:d2plus1_flow} there exists $\ep$, such that if $\| w_j\|_{L^1} < \ep$, then $\| \ps_j - \on{Id} \|_{\infty} < \de$, i.e., for all $(t,x) \in I \x \R^d$ we have $| \ps(t,x) - x | < \de$; in particular this implies $\ps_j(t)(U_j) \subseteq U_j + B_\de(0)$ and thus $\ol{\ps_j(t)(U_j)} \subseteq V_j$. Using \eqref{eq:norms_equivalent_M} we can bound $\| w_j \|_{L^1}$ via a bound on $\| u \|_{L^1}$.
\end{proof}

\begin{lemma}
\label{lem:phi_well_defined}
Let $s > d/2+1$ and $(\mc U_J, \mc V_J, \on{Id})$ be a fine cover of $M$ with respect to $\on{Id}$ with $\ch_j = \et_j|_{\mc U_j}$. With $\ep$ as in Lem. \ref{lem:localize_flow}, take a vector field $u$ with $\| u \|_{L^1(I, \mf X^s(M))} < \ep$ and define $\ph(t)$ via \eqref{eq:def_phi_on_M}. Then $\ph(t)$ is well-defined and $\ph(t) \in \mc D^s(M)$ for all $t \in I$.
\end{lemma}

\begin{proof}
To show that $\ph(t)$ is well-defined we need to show that whenever $\mc U_i \cap \mc U_j \neq \emptyset$, we have on the intersection the identity
\[
\et_i\i \o \ps_i(t) \o \et_i = \et_j\i \o \ps_j(t) \o \et_j\,.
\]
Omitting the argument $t$, we note that the identity $T\et_i \o u = v_i \o \et_i$ means that $u$ is $\et_i$-related to $v_i$, i.e., $u \sim_{\et_i} v_i$; hence on $\et_i(\mc U_i \cap \mc U_j)$ we have the relation $u_i \sim_{\et_j \o \et_i\i} u_j$, implying for the flows the identity
\[
\et_j\o \et_i\i \o \ps_i(t) = \ps_j(t) \o \et_j \o \et_i\i \,,
\]
and thus showing the well-definedness of $\ph(t)$. From \eqref{eq:def_phi_on_M} we see that $\ph(t) \in H^s(M,M)$, that $\ph(t)$ is invertible and that $\ph\i(t) \in H^s(M,M)$ as well. Thus $\ph(t) \in \mc D^s(M)$.
\end{proof}

The following lemma is a generalization of Lem. \ref{CompositionsLemma} to manifolds. Its main use will be when reformulated as a local equivalence of inner products in Sect. \ref{sec:strong_metrics}.

\begin{lemma}
\label{lem:composition_bound_M}
Let $s> d/2+1$ and $0 \leq s' \leq s$. Given $r > 0$ there exists a constant $C$, such that the inequality
\begin{equation}
\label{eq:composition_bound_M}
\|  v \o \ph \|_{H^{s'}} \leq C \| v \|_{H^{s'}}\,,
\end{equation}
holds for all $\ph \in D^s(M)$ that can be writted as $\ph = \on{Fl}_1(u)$ with $\| u\|_{L^1} < r$ and all $v \in H^{s'}(M)$ or $v \in \mf X^{s'}(M)$.
\end{lemma}

\begin{proof}
Choose a fine cover $(\mc U_I, \mc V_I, \on{Id})$ of $M$ with respect to $\on{Id}$ with $\ch_i = \et_i|_{\mc U_i}$. Let $\ep >0$ be such that if $\ph= \on{Fl}(u)$ with $\| u \|_{L^1} < \ep$ then $\ph \in \mc O^s(\mc U_I, \mc V_I)$. Such an $\ep$ exists, because $\mc O^s$ is open in $\mc D^s(M)$ and $\on{Fl}_1$ is continuous. We will show the inequality \eqref{eq:composition_bound_M} first for $r \leq \ep$.

Given $\ph = \on{Fl}_1(u)$ with $\| u \|_{L^1} < \ep$, define $\ph_i = \et_i \o \ph \o \et_i\i$ and $u_i = T\et_i \o u \o \et_i\i$, the extensions $\tilde u_i = E_{V_i}u_i$ and their flows $\tilde \ph_i = \on{Fl}_1(\tilde u_i)$. Given $f \in H^{s'}(M)$, the norm $\| f \o \ph\|_{H^{s'}(M)}$ is equivalent to
\[
\| f \o \ph\|_{H^{s'}(M)} \sim \sum_{i \in I} \left\| \left(f \o \ph\right)_i \right\|_{H^{s'}(U_i)}
\]
with $(f \o \ph)_i = f \o \ph \o \et_i\i$. Setting $f_i = f \o \et_i$, since $\ph \in \mc O^s$, we have the equality $(f \o \ph)_i = f_i \o \ph_i = E_{V_i}f_i \o \tilde \ph_i$ on $U_i$ and thus
\[
\left\| \left(f \o \ph\right)_i \right\|_{H^{s'}(U_i)} \leq
\left\| E_{V_i} f_i \o \tilde \ph_i \right\|_{H^{s'}(\R^d)} \leq
C_1 \| E_{V_i} f_i \|_{H^{s'}(\R^d)} \leq
C_2 \| f_i \|_{H^{s'}(V_i)}\,.
\]
The constant $C_1$ arises from Lem. \ref{lem:composition_bound_Rd}, since all $\tilde \ph_i$ are generated by vector fields with bounded norms. For $v \in \mf X^{s'}(M)$ the proof proceeds in the same way.

When $r > \ep$, we use the decomposition in Rem. \ref{rem:decomposition_principle} to write
\[
\ph = \ph^1 \o \ph^2 \o \dots \o \ph^N
\]
with $\ph^k \in \mc D^s(M)$, where $\ph^k = \on{Fl}_1(u^k)$ with $\| u^k\|_{L^1} < \ep$. Since $N$, the number of elements in the decomposition, depends only on $r$, the inequality \eqref{eq:composition_bound_M} can be shown inductively for $r$ of any size.
\end{proof}

To formulate the next lemma we need to introduce the geodesic distance of a right-invariant Riemannian metric on $\mc D^s(M)$. Fixing an inner product on $\mf X^s(M)$, we define
\[
\on{dist}^s(\ph, \ps) = \inf 
\left\{ \| u \|_{L^1([0,1], \mf X^s(M))} \,:\, \ps = \on{Fl}_1(u) \o \ph \right\}\,.
\]
See Sect. \ref{sec:strong_metrics} where it is shown, that $\on{dist}^s$ is indeed the geodesic distance associated to a Riemmnian metric and Sect. \ref{sec:completeness}, where it is shown, that the infimum is attained.

\begin{lemma}
\label{lem:local_lipschitz_M}
Let $s > d/2+1$. Given a fine cover $(\mc U_I, \mc W_I, \on{Id})$ of $M$ with respect to $\on{Id}_M$ with $\ch_i = \et_i$, there exists an $\ep>0$ and a constant $C$, such that for $\ph \in \mc D^s(M)$,
$\on{dist}^s(\on{Id}, \ph) < \ep$ implies $\ph \in \mc O^s(\mc U_I, \mc W_I)$ and such that the inequality
\[
\sum_{i \in I} \| \ph_i - \ps_i \|_{H^s(U_i)} \leq C\on{dist}^s(\ph, \ps)
\]
holds for all $\ph, \ps \in \mc D^s(M)$ inside the metric $\ep$-ball around $\on{Id}$ in $\mc D^s(M)$; here $\ph_i = \et_i \o \ph \o \et_i\i$ denotes the coordinate expression of $\ph$.
\end{lemma}

\begin{proof}
Choose first an intermediate cover $\mc V_I = (\mc V_i)_{i \in I}$, such that both $(\mc U_I, \mc V_I, \on{Id})$ and $(\mc V_I, \mc W_I, \on{Id})$ are fine covers of $M$ w.r.t. $\on{Id}$ and they all use the same coordinate charts $\et_i$. This implies in particular the inclusions $\ol{\mc U_i} \subseteq \mc V_i$ and $\ol{\mc V_i} \subseteq \mc W_i$. Let $\ep>0$ be such that
\[
\on{dist}^s(\on{Id}, \ph) < 3\ep
\quad \Rightarrow \quad
\ph \in \mc O^s(\mc U_I, \mc V_I)\text{ and }
\ph \in \mc O^s(\mc V_I, \mc W_I)\,.
\]
Note that since $\on{dist}^s(\on{Id}, \ph) = \on{dist}^s(\on{Id}, \ph\i)$, the same holds for $\ph\i$.

Let $\ph^1, \ph^2$ be inside the metric $\ep$-ball around $\on{Id}$ in $\mc D^s(M)$. Then
\[
\on{dist}^s(\ph^1, \ph^2) \leq 
\on{dist}^s(\ph^1, \on{Id}) + \on{dist}^s(\on{Id}, \ph^2) < 2 \ep\,.
\]
Let $v$ be a vector field with $\on{Fl}_1(v) = \ph^2 \o (\ph^1)\i$ and $\| v\|_{L^1} < 2\ep$. Denote its flow by $\ps(t) = \on{Fl}_t(v)$. Then
\[
\on{dist}^s(\on{Id}, \ps(t)) \leq
\on{dist}^s(\on{Id}, \ph^1) + \on{dist}^s(\ph^1, \ps(t)) < 3\ep\,,
\]
and thus $\ps(t) \in \mc O^s(\mc V_I, \mc W_I)$. Define $v_i(t) = T\et_i \o v(t) \o \et_i\i$ and $\ps_i(t) = \et_i \o \ps(t) \o \et_i\i$. Then $v_i(t) \in \mf X^s(W_i)$ and the following equality holds
\begin{equation}
\label{eq:flow_equality_Vi}
\left(\ph^2 \o (\ph^1)\i\right)_i(x) - x = \int_0^1 v_i(t, \ps_i(t, x)) \ud t
\quad \text{ for } x \in V_i\,.
\end{equation}
Because $\ph^1, (\ph^1)\i, \ph^2 \o (\ph^1)\i \in \mc O^s(\mc V_I, \mc W_I)$ we have
\begin{equation}
\label{eq:split_up_i_On_Vi}
\left(\ph^2 \o (\ph^1)\i\right)_i(x) =
\ph^2_i \o (\ph^1_i)(x)
\quad \text{ for } x \in V_i\,,
\end{equation}
and since $\ph^1 \in \mc O^s(\mc U_I, \mc V_I)$, equality \eqref{eq:flow_equality_Vi} together with \eqref{eq:split_up_i_On_Vi} implies
\begin{equation}
\label{eq:flow_equality_Ui}
\ph_i^2(x) - \ph^1_i(x) = \int_0^1 v_i(t) \o \ps_i(t) \o \ph^1_i(x) \ud t
\quad \text{ for } x \in U_i\,.
\end{equation}
Note that the domain, where the equality holds, has shrunk from $V_i$ to $U_i$. This is the reason for introducing the intermediate cover $\mc V_I$.

Since $\on{dist}^s(\on{Id}, \ph^1) < \ep$, we can write $\ph^1 = \on{Fl}_1(u^1)$ for a vector field $u^1$ with $\| u^1\|_{L^1} < \ep$. Set $\ph(t) = \on{Fl}_t(u^1)$. Introduce the coordinate expressions $u^1_i = T\et_i \o u^1 \o \et_i\i$, extend them to $\tilde u^1_i = E_{W_i}u^1_i$ and denote their flows by $\tilde \ph_i(t) = \on{Fl}_t(\tilde u_i)$. Since $\on{dist}^s(\on{Id}, \ph(t)) < \ep$, it follows that $\ph(t) \in \mc O^s(\mc U_I, \mc V_I)$ and thus $\ph_i(t,x) = \tilde \ph_i(t,x)$ for $x \in U_i$; in particular $\ph^1_i = \tilde \ph_i(1)$ on $U_i$.

Similarly we define the extension $\tilde v_i = E_{W_i} v_i$ and its flow $\tilde \ps_i(t) = \on{Fl}_t(\tilde v_i)$ and by the same argument we obtain $\ps_i(t, x) = \tilde \ps_i(t,x)$ for all $t$ and $x \in V_i$. The advantage is, that $\tilde \ph_i(1)$ and $\tilde \ps_i(t)$ are defined on all of $\R^d$ and are elements of $\mc D^s(\R^d)$. Thus \eqref{eq:flow_equality_Ui} can be written as
\begin{equation*}
\ph^2(x) - \ph^1_i(x) = \int_0^1 \tilde v_i(t) \o \tilde \ps_i(t) \o \tilde \ph_i(1)(x) \ud t
\quad \text{ for } x \in U_i\,,
\end{equation*}
and we can estimate
\begin{multline}
\label{eq:lip_estimate_vector_field_M}
\| \ph^2_i - \ph^1_i \|_{H^s(U_i)} \leq
\int_0^1 \left\| \tilde v_i(t) \o \tilde \ps_i(t) \o \tilde \ph_i(1) 
\right\|_{H^s(\R^d)} \ud t \leq \\
\leq C_1 \int_0^1 \left\| \tilde v_i(t) \right\|_{H^s(\R^d)} \ud t \leq
C_2 \| v \|_{L^1([0,1], \mf X^s(M))}\,.
\end{multline}
The constant $C_1$ appears from invoking Lem. \ref{lem:composition_bound_Rd}, since both $\tilde \ph_i$ and $\tilde \ps_i$ are generated by vector fields with bounded $L^1$-norms. Since $v$ was taken to be any vector field with $\on{Fl}_1(v) = \ph^2 \o (\ph^1)\i$, we can take the infimum over $v$ in \eqref{eq:lip_estimate_vector_field_M} to obtain
\[
\| \ph^1_i - \ph^2_i \|_{H^s(U_i)} \leq C_2\on{dist}^s(\ph^1, \ph^2)\,,
\]
from which the statement of the lemma easily follows.
\end{proof}

\section{Riemannian metrics on $\mc D^s(M)$}
\label{sec:strong_metrics}

\subsection{Strong metrics}

Let $(M,g)$ be $\R^d$ with the Euclidean metric or a closed $d$-dimensional Riemannian manifold and $s > d/2+1$. On the diffeomorphism group $\mc D^s(M)$ we put a right-invariant Sobolev metric $G^s$ of order $s$, defined at the identity by
\begin{equation}
\label{eq:general_sobolev_inner_product}
\langle u, v \rangle_{H^s} = \int_M g(u, Lv) \ud \mu\,,
\end{equation}
for $u, v \in \mf X^s(M)$, where $L \in OPS^{2s}_{1,0}$ is a positive, self-adjoint, elliptic operator of order $2s$. By right-invariance the metric is given by
\begin{equation}
\label{eq:def_rinv_metric_Ds}
G_\ph^s(X_\ph,Y_\ph) = \langle X_\ph \o \ph\i, Y_\ph\o \ph\i \rangle_{H^s} \,,
\end{equation}
for $X_\ph, Y_\ph \in T_\ph \mc D^s(M)$. Since $\mc D^s(M)$ is a topological group, the metric $G^s$ is a continuous Riemannian metric.

When $s=n$ is an integer and the operator is
\[
L=(\on{Id} + \De^n) \text{ or }
L=(\on{Id} + \De)^n\,,
\]
where $\De u = (\de du^\flat + d\de u^\flat)^\sharp$ is the positive definite Hodge Laplacian or some other combination of intrinsically defined differential operators with smooth coefficient functions, then one can show that the metric $G^n$ is in fact smooth on $\mc D^n(M)$. Since the inner products $G^n$ generate the topology of the tangent spaces, this makes $(\mc D^n(M), G^n)$ into a strong Riemannian manifold; see \cite{Ebin1970} and \cite{Misiolek2010} for details and \cite{Lang1999} for infinite-dimensional Riemannian geometry for strong metrics.

The existence of strong metrics is somewhat surprising, since there is a result by Omori \cite{Omori1978} stating that there exist no infinite-dimensional Banach Lie groups acting effectively, transitively and smoothly on a compacts manifold. $\mc D^s(M)$ acts effectively, transitively and smoothly on $M$. While $\mc D^s(M)$ is not a Lie group, but only a topological group with a smooth right-multiplication, the definition \eqref{eq:def_rinv_metric_Ds} of the metric uses the inversion, which is only a continuous operation. As it turns out one can have a smooth, strong, right-invariant Riemannian metric on a topological group, that is not a Lie group.

\begin{remark}
\label{rem:admissible_norms}
Most results in this paper -- in particular the existence and continuity of flow maps and estimates on the composition -- depend only on the topology of the Sobolev spaces and are robust with respect to changes to equivalent inner products. The smoothness of the metric does not fall into this category. Assume $\langle\cdot,\cdot \rangle_1$ and $\langle\cdot,\cdot \rangle_2$ are two equivalent inner products on $\mf X^s(M)$ and denote by $G^1$ and $G^2$ the induced right-invariant Riemannian metrics on $\mc D^s(M)$. Then the smoothness of $G^1$ does not imply anything about the smoothness of $G^2$. To see this, factorize the map $(\ph, X, Y) \mapsto G_\ph(X, Y)$ into
\[
\arraycolsep=2pt
\begin{array}{ccccc}
T\mc D^s \x_{\mc D^s} T\mc D^s &\to
&\mf X^s  \x \mf X^s &\to & \R \\
(\ph, X, Y) &\mapsto &(X \o \ph\i, Y \o \ph\i) &\mapsto
& \langle X \o \ph\i, Y \o \ph\i \rangle
\end{array}\,.
\]
Changing the inner product corresponds to changing the right part of the diagram. However the left part of the diagram is not smooth by itself, i.e., the map $(\ph, X) \mapsto X \o \ph\i$ is only continuous. The smoothness of the Riemannian metric is thus a property of the composition.
\end{remark}

\begin{openquestion*}
What class of inner products on $\mf X^s(M)$ induces smooth right-invariant Riemannian metrics on $\mc D^s(M)$? Does this hold for all $s > d/2+1$, non-integer, and all metrics of the form \eqref{eq:general_sobolev_inner_product}?
\end{openquestion*}

\subsection{Geodesic distance}
Given a right-invariant Sobolev metric $G^s$, the induced geodesic distance is
\[
\on{dist}^s(\ph,\ps) = \inf \left\{ \mc L(\et) \,:\, 
\et(0) = \ph, \et(1) = \ps \right\}\,,
\]
with the length functional
\[
\mc L(\et) = \int_0^1 \sqrt{G_{\et(t)}\left(\p_t \et(t), \p_t \et(t)\right)} \ud t\,,
\]
and the infimum is taken over all piecewise smooth paths. Due to right-invariance we have
\[
\mc L(\et) = \| \p_t \et \o \et\i\|_{L^1([0,1], \mf X^s(M))}\,,
\]
where $\mf X^s(M)$ is equipped with the inner product $\langle \cdot,\cdot \rangle_{H^s}$. Since piecewise smooth paths are dense in $L^1$ one can also compute the distance via
\[
\on{dist}^s(\ph, \ps) = \inf 
\left\{ \| u \|_{L^1([0,1], \mf X^s(M))} \,:\, \ps = \on{Fl}_1(u) \o \ph \right\}\,.
\]
It was shown in Thms. \ref{thm:d2plus1_flow} and \ref{thm:d2plus1_flow_M} that the flow-map is well-defined. To define the geodesic distance a continuous Riemannian metric is sufficient and thus the following results hold for $s > d/2+1$.

\subsection{Uniform equivalence of inner products}
Since the open geodesic ball around $\on{Id}$ of radius $r$ coincides with the set
\[
\left\{ \on{Fl}_1(u)\,:\, \| u\|_{L^1([0,1],\mf X^s(M))} < r \right\} =
\left\{ \ph \,:\, \on{dist}^s(\on{Id}, \ph) < r \right\}\,,
\]
we can reformulate Lem. \ref{lem:composition_bound_Rd} and Lem. \ref{lem:composition_bound_M} as follows.

\begin{corollary}
\label{cor:composition_bound_geodesic}
Let $s> d/2+1$ and $0 \leq s' \leq s$. Given $r > 0$ there exists a constant $C$, such that the inequality
\begin{equation*}
\|  v \o \ph \|_{H^{s'}} \leq C \| v \|_{H^{s'}}\,,
\end{equation*}
holds for all $\ph \in \mc D^s(M)$ with $\on{dist}^s(\on{Id}, \ph) < r$ and all $v \in H^{s'}(M)$ or $v \in \mf X^{s'}(M)$.
\end{corollary}

Since $\on{dist}^s(\on{Id}, \ph) = \on{dist}^s(\on{Id}, \ph\i)$, we have for some constant $C$ on every geodesic ball the inequalities
\[
C\i \| v \|_{H^s} \leq \|  v \o \ph\i \|_{H^{s}} \leq C \| v \|_{H^{s}}\,,
\]
stating that the inner products induced by $G^s(\cdot,\cdot)$ is equivalent to the inner product $\langle \cdot ,\cdot \rangle_{H^s}$ on every geodesic ball with a constant that depends only on the radius of the ball. 

This result enables us to prove that on $\R^d$ the $\mf X^s(\R^d)$-norm is Lipschitz with respect to the geodesic distance on any bounded metric ball. We will use this lemma to show that the geodesic distance is a complete metric.

\begin{lemma}
\label{lem:local_lipschitz_Rd}
Let $s > d/2+1$. Given $r>0$, there exists a constant $C$, such that the inequality
\[
\| \ph_1 - \ph_2 \|_{H^s} \leq C \on{dist}^s(\ph_1, \ph_2)\,,
\]
holds for all $\ph_1, \ph_2 \in \mc D^s(\R^d)$ with $\on{dist}^s(\on{Id}, \ph_i) < r$.
\end{lemma}

\begin{proof}
We have
\[
\on{dist}^s(\ph_1, \ph_2) \leq 
\on{dist}^s(\ph_1, \on{Id}) + \on{dist}^s(\on{Id}, \ph_2) < 2 r\,.
\]
Let $u$ be a vector field with $\ph_2 = \on{Fl}_1(u) \o \ph_1$ and $\| u\|_{L^1} < 2r$. Denote its flow by $\ps(t) = \on{Fl}_t(u)$. Then
\[
\on{dist}^s(\on{Id}, \ps(t)) \leq
\on{dist}^s(\on{Id}, \ph_1) + \on{dist}^s(\ph_1, \ps(t)) < 3r\,,
\]
and thus using Cor. \ref{cor:composition_bound_geodesic} there exists a constant $C$, allowing us to estimate
\[
\| \ph_1 - \ph_2 \|_{H^s} \leq
\int_0^1 \| u(t) \o \ps(t) \o \ph_1 \|_{H^s} \ud t
\leq
C \int_0^1 \| u(t) \|_{H^s} \ud t\,.
\]
By taking the infimum over all vector fields we obtain the result.
\end{proof}

On an arbitrary compact manifold $M$ we can show only a local version of Lem. \ref{lem:local_lipschitz_Rd}, which we did in Lem. \ref{lem:local_lipschitz_M}. This local version will however be enough to show metric completeness.

\section{Completeness of diffeomorphism groups}
\label{sec:completeness}

In this section we will combine the results on flows of $L^1$-vector fields and estimates on the geodesic distance, to show that $\mc D^s(M)$ with a Sobolev-metric $G^s$ of order $s$ is a complete Riemannian manifold in all the senses of the theorem of Hopf--Rinow.

The completeness results are valid for the class of metrics satisfying the following hypothesis:
\begin{equation}
\label{eq:hyp_strong}
\tag{H} 
\parbox{11cm}{
Let $M$ be $\R^d$ or a closed manifold and let $\langle\cdot,\cdot\rangle_{H^s}$ be an inner product on $\mf X^s(M)$, such that the induced right-invariant metric
\begin{equation*}
G_\ph^s(X_\ph,Y_\ph) = \langle X_\ph \o \ph\i, Y_\ph\o \ph\i \rangle_{H^s} \,,
\end{equation*}
on $\mc D^s(M)$ is smooth, thus making $(\mc D^s(M), G^s)$ into a strong Riemannian manifold.
}
\end{equation}

As discussed in Sect. \ref{sec:strong_metrics}, this hypothesis is satisfied for a large class of Sobolev metrics of integer order.

First we show the existence of minimizing geodesics between any two diffeomorphisms in the same connected component. This extends Thm. 9.1 in \cite{Misiolek2010}, where existence of minimizing geodesics was shown only for an open and dense subset.

This existence result is shown using the direct method of the calculus of variations. Namely, the variational problem we consider consists of the minimization of an energy which is, under a change of variables, a weakly lower semi-continuous functional on a weakly closed constraint set. The change of variables is simply given by the vector field associated with the path and in the next lemma, we also prove that the constraint set is weakly closed.

\begin{lemma} \label{WeakConvergenceFlows}
Let $\psi_0, \psi_1 \in \mc D^s(M)$ be two diffeomorphisms and define
\[
\Om_{\ps_0} H^1 = \left\{ \ph \,:\, \ph(0) = \ps_0, \right\}
\subseteq H^1([0,1], \mc D^s(M))\,
\]
as well as
\[
\Om_{\ps_0,\ps_1} H^1 = \left\{ \ph \,:\, \ph(0) = \ps_0,\,\ph(1) = \ps_1 \right\}
\subseteq \Om_{\ps_0} H^1 \,
\]
which are submanifolds of the manifold $H^1([0,1], \mc D^s(M))$ of $H^1$-curves with values in $\mc D^s(M)$.
The map 
\begin{equation*}
\Theta :
\, \Om_{\ps_0} H^1 \to  L^2([0,1],\mf X^s(M))\,,\quad
 \ph \mapsto  \left( t \mapsto \partial_t \ph(t) \circ \ph(t)^{-1}\right)
\end{equation*}
is a homeomorphism for the strong topologies and the set $
\Theta(\Om_{\ps_0,\ps_1} H^1)$ is closed 
with respect to the weak topology on $L^2([0,1],\mf X^s(M))$.
\end{lemma}

\begin{proof}
The definition of $\Theta$ is a direct consequence of Lem.~\ref{CompositionsLemma}. The inverse of $\Theta$ is given by the flow with initial condition $\ph(0) = \ps_0$, $\Th\i(u) = \left( t \mapsto \on{Fl}_t(u) \o \ps_0 \right)$. The flow belongs to $H^1([0,1],  \mc D^s(M))$ by Thm.~\ref{thm:d2plus1_flow} for $M = \R^d$ and by Thm.~\ref{thm:d2plus1_flow_M} for $M$ a closed manifold.

We now prove the second part of the lemma in the case $M = \R^d$. Consider a sequence $u^n \in L^2([0,1], H^s(\R^d,\R^d))$, converging weakly to $u$. Denote by $\ph^n$ and $\ph$ the respective flows. We will show that $\ph^n(t,x) \to \ph(t,x)$ pointwise in $x$ and uniformly in $t$. Because $s > d/2+1$, we have the continuous embedding $H^s(\R^d,\R^d) \hookrightarrow C^1_b(\R^d,\R^d)$, where $C^1_b$ denotes the space of $C^1$-functions with bounded derivatives, and we let $C > 0$ be such that $\| u \|_{C^1_b} \leq C \| u \|_{H^s}$ holds for all $u \in H^s$.

Take $(t,x) \in [0,1] \x \R^d$. Then
\begin{equation}\label{StandardIneq}
\begin{aligned}
| \ph^n&(t, x) - \ph(t, x) |  
\leq \left| \int_0^t  u^n(\ta, \ph^n(\ta, x)) - u(\ta, \ph(\ta, x))   \ud \ta \right| \\ 
& \leq \int_0^t | u^n(\ta, \ph^n(\ta, x)) - u^n(\ta, \ph(\ta, x) | \ud \ta
+ \left| \int_0^t u^n(\ta, \ph(\ta, x)) - u(\ta, \ph(\ta, x) \ud \ta \right| \,.
\end{aligned}
\end{equation}
For the first term we have
\begin{align*}
\int_0^t | u^n(\ta, \ph^n(\ta, x)) - u^n(\ta, \ph(\ta, x) | \ud s
&\leq \int_0^t \| u^n(\ta) \|_{C^1_b} \left| \ph^n(\ta, x) - \ph(\ta, x) \right| \ud \ta \\
&\leq \int_0^t C \| u^n(\ta) \|_{H^s} \left| \ph^n(\ta, x) - \ph(\ta, x) \right| \ud \ta\,.
\end{align*}

The second term can be written as $\left| \langle m_{t,x}, u^n - u \rangle \right|$, where
\[
\langle m_{t,x}, v \rangle = \int_0^t v(\ta, \ph(\ta, x)) \ud \ta\,,
\]
which is a linear map $m_{t,x} : L^2([0,1], H^s) \to \R^d$. Fix $x \in \R^d$ and consider the functions
\[
m^n: [0,1] \to \R^d\,,\, t \mapsto \langle m_{t, x}, u^n \rangle 
\]
They converge pointwise $m^n(t) = \langle m_{t,x}, u^n \rangle \to \langle m_{t,x}, u \rangle = m(t)$ for each $t \in [0,1]$. Because $u^n \rightharpoonup u$ weakly, the sequence $(u^n)_{n \in \mathbb N}$ is bounded in $L^2([0,1],H^s)$ and hence the following estimates show that the sequence $(m^n)_{n \in \mathbb N}$ is equicontinuous:
\begin{align*}
\left| \langle m_{t,x} - m_{r,x}, u^n \rangle \right|
&\leq \left| \int_r^t u^n(\ta, \ph(\ta, x)) \ud \ta \right|
\leq C \sqrt{|t-r|} \| u^n \|_{L^2([0,1], H^s)}\,.
\end{align*}
By Arzela-Ascoli it follows, that $\langle m_{t,x}, u^n \rangle \to \langle m_{t,x}, u \rangle$ uniformly in $t$.

Going back to \eqref{StandardIneq}, we define $A(t) = | \ph^n(t, x) - \ph(t, x) |$ and we have the estimate
\[
A(t) \leq \int_0^t C \| u^n(\ta) \|_{H^s} A(\ta) \ud \ta 
+ \left|\langle m_{t,x}, u^n -u \rangle\right|\,.
\]
Gronwall's inequality then leads to
\begin{align*}
| \ph^n(t, x) - \ph(t, x) |
&\leq \left|\langle m_{t,x}, u^n -u \rangle\right| + \\
&\,\,{}+C \int_0^t \left|\langle m_{\ta,x}, u^n -u \rangle\right| \| u^n(\ta)\|_{H^s}
\exp\left( C \|u^n\|_{L^1([0,1], H^s)} \right) \ud \ta\,.
\end{align*}
The uniform convergence of $\langle m_{\ta, x}, u^n-u \rangle \to 0$ shows that $\ph^n(t,x) \to \ph(t,x)$ pointwise in $x$ and uniformly in $t$.

Now consider a sequence of paths in $\ph^n \in \Om_{\ps_0,\ps_1} H^1$ such that $u^n = \Theta(\ph^n)$ converges weakly to $u = \Theta(\ph)$. We have to show that $\ph(1) = \ps_1$. We have $\ph^n(1) = \ps_1$ for all $n \in \mathbb N$ and using the pointwise convergence of the flow established above, also $\ph(1, x) = \lim_{n \to \infty} \ph^n(1,x) = \lim_{n \to \infty} \ps_1(x) = \ps_1(x)$. This concludes the proof for $M = \R^d $.

When $M$ is a compact manifold the result follows by reduction to $\R^d$ and the use of a fine cover.
\end{proof}

\begin{theorem}
\label{thm:geodesic_bvp}
Let $(\mc D^s(M), G^s)$ satisfy hypothesis \eqref{eq:hyp_strong}. Then any two elements of $\mc D^s(M)_0$ can be joined by a minimizing geodesic.
\end{theorem}

\begin{proof}
Let $\ps_0, \ps_1 \in \mc D^s(M)_0$ be two diffeomorphisms. Our aim is to minimize
\begin{equation}
\label{eq:energy_klingenberg}
\mc E(\ph) = \int_0^1 G_{\ph(t)}\left(\p_t \ph(t), \p_t \ph(t)\right) \ud t\,,
\end{equation}
on 
$\Om_{\ps_0,\ps_1} H^1$. We have,
\begin{equation*}
\mc E(\ph) = \int_0^1 \| \Theta(\ph) \|^2_{H^s} \ud t 
= \left\| \Th(\ph) \right\|^2_{L^2([0,1], \mf X^s)}\,.
\end{equation*}
Consider a minimizing sequence $\ph^n \in \Om_{\ps_0,\ps_1} H^1$, thus $\Theta(\ph^n) \in L^2([0,1],\mf X^s)$ is bounded and after extraction of a subsequence, we can assume that $\Theta(\ph^n)$ weakly converges to $\Theta(\ph^\ast)$. Lemma~\ref{WeakConvergenceFlows} ensures that $\ph^\ast \in \Om_{\ps_0,\ps_1} H^1$. Because the norm on $L^2([0,1], \mf X^s)$ is sequentially weakly lower semi-continuous, we have $\mc E(\ph^\ast) \leq \liminf \mc E(\ph_n)$. Thus $\ph^\ast$ is a minimizer of $\mc E$.

To show regularity of minimizers, we consider $\mc E$ given by~\eqref{eq:energy_klingenberg} as a functional on the space $H^1([0,1], \mc D^s(M))$. This functional is differentiable and the derivative is given by
\[
D \mc E(\ph).h = \int_0^1 G_{\ph(t)}(\p_t \ph(t), \nabla_{\p_t \ph(t)} h(t)) \ud t\,,
\]
with $\nabla$ denoting the covariant derivative of the metric $G$ \cite[Thm.~2.3.20]{Klingenberg1995}. The minimizer $\ph^\ast$ constructed above is thus a critical point of $\mc E$. By standard bootstrap methods it follows that critical points are smooth in time and thus solutions of the geodesic equation, e.g., it is shown in \cite[Lem.~2.4.3]{Klingenberg1995} that critical points of $\mc E$, restricted to paths with fixed endpoints, are geodesics on the underlying manifold.\footnote{In \cite{Klingenberg1995} the space of paths, $H^1([0,1], M)$, is constructed only for finite-dimensional manifolds $M$. However the results, that are necessary for us, remain valid with the same proofs, when $M$ is a strong Riemannian manifold modelled on a separable Hilbert space. The important part is that $[0,1]$ is finite dimensional and compact.}

\end{proof}

\begin{remark}
\label{rem:geodesic_bvp_subgroup}
Let $M$ and $\left(\mc D^s(M), G^s\right)$ satisfy the assumptions of Thm. \ref{thm:geodesic_bvp}. The same proof can be used to show the existence of minimizing geodesics for subgroups of the diffeomorphism group: the group $\mc D^s_{\mu}(M)$ of diffeomorphisms preserving a volume form $\mu$ or the group $\mc D^s_{\om}(M)$ of diffeomorphisms preserving a symplectic form $\om$. In fact the proof can be generalized to any closed, connected subgroup $\mc C$, that is also a Hilbert submanifold of $\mc D^s(M)$ since $T_{\on{Id}}\mc C$ is a closed Hilbert subspace of $\mf X^s$. Then $L^2([0,1],T_{\on{Id}}\mc C)$ is a closed subspace of $L^2([0,1],\mf X^s)$ and thus weakly closed. Therefore, the limit found in the proof will satisfy the boundary conditions and will also belong to $\mc C$.

\end{remark}

Next we show that the the group of diffeomorphisms with the induced geodesic distance is a complete metric space. There is a related result by Trouv\'e -- see \cite[Thm. 8.15]{Younes2010} -- which shows metric completeness for the groups of diffeomorphisms $\mc G_{\mc H}$, generated by an admissible space of vector fields $\mc H$; see Sect. \ref{sec:lddmm_framework} for details. Since we obtain $\mc D^s(\R^d)_0 = \mc G_{H^s(\R^d,\R^d)}$ in Thm. \ref{thm:equivalence}, this provides another proof of metric completeness of $\mc D^s(\R^d)_0$.

\begin{theorem}
\label{thm:diff_M_metric_completeness}
Let $\left(\mc D^s(M), G^s\right)$ satisfy hypothesis \eqref{eq:hyp_strong}. Then $\left(\mc D^s(M)_0, \on{dist}^s\right)$ is a complete metric space.
\end{theorem}

\begin{proof}
{\bfseries Case: $M=\R^d$.}
Consider first the case $M=\R^d$. Let $\ep > 0$ be such that $\on{Id}+ B_\ep(0) \subset \mc D^s(\R^d)$, where $B_\ep(0)$ is the $\ep$-ball in $H^s(\R^d,\R^d)$. By Cor. \ref{cor:composition_bound_geodesic} there exists a constant $C$, such that the inequality
\begin{equation}
\label{eq:transfer_cauchy_Rd}
\| \ph - \ps \|_{H^s} \leq C \on{dist}^s(\ph, \ps)
\end{equation}
holds on the metric $\ep$-ball around $\on{Id}$ in $\mc D^s(\R^d)$.

Let $(\ph^n)_{n \in \mb N}$ be a Cauchy sequence in $\mc D^s(\R^d)_0$. We can assume without loss of generality that $\on{dist}^s(\ph^n, \ph^m) < \tfrac 12 \ep/C$ holds for all $n,m \in \mb N$ and since the distance is right-invariant we can also assume that $\ph^1 = \on{Id}$. Then \eqref{eq:transfer_cauchy_Rd} shows, that $\left(\on{Id} - \ph^n\right)_{n \in \mb N}$ is a Cauchy sequence in $H^s(\R^d,\R^d)$. Denote the limit by $\on{Id} - \ph^\ast$. From
\[
\| \on{Id} - \ph^\ast \|_{H^s} = \| \ph^1 - \ph^\ast \|_{H^s} \leq 
C \limsup_{n \to \infty} \on{dist}^s(\ph^1, \ph^n) \leq \tfrac 12 \ep
\]
it follows that $\ph^\ast \in \mc D^s(\R^d)$ and since the manifold topology coincides with the metric topology, we also have $\on{dist}^s(\ph^n, \ph^\ast) \to 0$. Thus $\mc D^s(\R^d)_0$ is complete.

{\bfseries Case: $M$ a closed manifold.}
The proof for a compact manifold proceeds in essentially the same way, the added complication is, that one has to work in a coordinate chart around the identity. Choose a fine cover $(\mc U_I, \mc V_I, \on{Id})$ of $M$ with respect to $\on{Id}$ such that $\et_i = \ch_i|_{\mc U_i}$. There exists $\ep_1 > 0$, such that if $\on{dist}^s(\on{Id}, \ph) < \ep_1$, then $\ph \in \mc O^s = \mc O^s(\mc U_I, \mc V_I)$. For 
$h \in \mc O^s \subseteq H^s(M,M)$ 
we define
\[
h_i = \et_i \o h \o \et_i\i,\, h_i \in \mc D^s(U_i,\R^d)\,.
\]
and by Lem. \ref{lem:local_lipschitz_M} there exists a constant $C$, such that the inequality
\begin{equation}
\label{eq:transfer_cauchy_M}
\| \ph_i - \ps_i \|_{H^s(U_i)} \leq C \on{dist}^s(\ph, \ps)
\end{equation}
is valid for all $i \in I$ and all $\ph, \ps \in \mc D^s(M)$ in the geodesic $\ep_1$-ball around $\on{Id}$. Furthermore, since $\mc D^s(M)$ is open in $H^s(M,M)$, there exists an $\ep_2>0$, such that
\begin{equation}
\label{eq:to_belong_to_diff_M}
h \in \mc O^s \text{ and } \| \on{Id} - h_i \|_{H^s(U_i)} < \ep_2,\, \forall i \in I \quad\Rightarrow\quad
h \in \mc D^s(M)\,.
\end{equation}

Given these preparations, let $(\ph^n)_{n \in \mb N}$ be a Cauchy sequence in $\mc D^s(M)_0$. We can assume w.l.o.g. that $\on{dist}^s(\ph^n, \ph^m) < \min(\ep_1, \tfrac 12 \ep_2/C)$ for all $n,m \in \mb N$ and because the distance is right-invariant also that $\ph^1 = \on{Id}$. It then follows from \eqref{eq:transfer_cauchy_M}, that for all $i \in I$, the sequences $(\ph^n_i)_{n \in \mb N}$ are Cauchy sequences in $H^s(U_i,\R^d)$. Denote their limits by $\ph_i^\ast$. Whenever $\mc U_i \cap U_j \neq \emptyset$, we have the compatibility conditions
\[
\et_i\i \o \ph^n_i \o \et_i = \et_j\i \o \ph^n_j \o \et_j
\quad\text{on }\mc U_i \cap \mc U_j\,,
\]
and since convergence in $H^s(U_i,\R^d)$ implies pointwise convergence, the compatibility conditions also hold for the limit $\ph_i^\ast$. Thus we can define a function $\ph^\ast$ on $M$ via $\ph^\ast|_{\mc U_i} = \et_i\i \o \ph^\ast_i \o \et_i$ and $\ph^n \to \ph^\ast$ in $H^s(M,M)$. We also have
\[
\| \on{Id} - \ph^n_i \|_{H^s(U_i)} \leq C \on{dist}^s(\on{Id}, \ph^n) \leq \tfrac 12 \ep_2\,,
\]
and so using \eqref{eq:to_belong_to_diff_M}, we see after passing to the limit that $\ph^\ast \in \mc D^s(M)$. As the manifold topology on $\mc D^s(M)_0$ coincides with the metric topology, it follows that $\on{dist}^s(\ph^n, \ph^\ast) \to 0$ and hence $\mc D^s(M)_0$ is a complete metric space.
\end{proof}

\begin{remark}
\label{rem:metric_completeness_subgroup}
Let $M$ and $\left(\mc D^s(M), G^s\right)$ satisfy the assumptions of Thm. \ref{thm:diff_M_metric_completeness}. Consider a closed, connected subgroup $\mc C$ and denote by $\on{dist}^s_{\mc C}$ the geodesic distance of the submanifold $(\mc C, G^s)$. Then $(\mc C, \on{dist}^s_{\mc C})$ is a complete metric space as well. This follows from the closedess of $\mc C$ and the inequality
$
\on{dist}^s(\ph, \ps) \leq \on{dist}^s_{\mc C}(\ph, \ps)\,,
$
which holds for all $\ph, \ps \in \mc C$.

Similar to Rem. \ref{rem:geodesic_bvp_subgroup} this applies in particular to the groups $\mc D^s_\mu(M)$ and $\mc D^s_\om(M)$ of diffeomorphisms preserving a given volume form or symplectic structure.
\end{remark}

We can now collect the various completeness properties diffeomorphism groups endowed with strong smooth Sobolev-type Riemannian metrics.

\begin{corollary}
\label{cor:diff_hopf_rinow}
Let $(\mc D^s(M), G^s)$ satisfy hypothesis \eqref{eq:hyp_strong}. Then
\begin{enumerate}
\item
$(\mc D^s(M), G^s)$ is geodesically complete.
\item
$(\mc D^s(M)_0, \on{dist}^s)$ is a complete metric space.
\item
Any two elements of $\mc D^s(M)_0$ can be joined by a minimizing geodesic.
\end{enumerate}
\end{corollary}

\begin{proof}
Geodesic completeness follows from metric completeness; see \cite{Lang1999}. It is also shown in \cite[Lem. 5.2]{GayBalmaz2015}, that every strong right-invariant metric on a manifold, that is a topological group with a smooth right-multiplication, is geodesically complete.

Metric completeness is shown in Thm. \ref{thm:diff_M_metric_completeness} and the existence of minimizing geo\-de\-sics in Thm. \ref{thm:geodesic_bvp}. For the statements about subgroups see Rems. \ref{rem:geodesic_bvp_subgroup} and \ref{rem:metric_completeness_subgroup}.
\end{proof}

Following Rem.~\ref{rem:geodesic_bvp_subgroup} and Rem.~\ref{rem:metric_completeness_subgroup} the methods of proof can also be applied to the subgroups $\mc D^s_\mu(M)$ and $\mc D^s_\om(M)$ of diffeomorphisms preserving a volume form $\mu$ or a symplectic structure $\om$.

\section{Applications to diffeomorphic image matching}
\label{sec:lddmm_framework}

\subsection{The group generated by an admissible vector space}

Let $(\mc H, \langle \cdot ,\cdot \rangle_{\mc H})$ be a Hilbert space of vector fields, such that the norm on $\mc H$ is stronger than the uniform $C^1$-norm, i.e., $\mc H \hookrightarrow C^1_b(\R^d,\R^d)$.
We call such an $\mc H$ an \emph{admissible vector space}.
This embedding implies that pointwise evaluations are continuous $\R^d$-valued forms on $\mc H$: for $x \in \R^d$, $\on{ev}_x : f \in \mc H \mapsto f(x) \in \R^d$ is continuous and $\on{ev}_x^v(f) := \langle f(x), v \rangle$ is a linear form on $\mc H$; here $v \in \R^d$ and $\langle \cdot, \cdot \rangle$ denotes the Euclidean scalar product on $\R^d$. Such a space is called a \emph{reproducing kernel Hilbert space} and is completely defined by its kernel. This kernel is defined as follows:
denoting $K: \mc H^* \to \mc H$ the Riesz isomorphism between $\mc H^*$ (the dual of $\mc H$) and $\mc H$, the reproducing  kernel of $\mc H$ evaluated at points $x,y \in \R^d$, denoted by $\mathsf{k}(x,y) \in L(\R^d,\R^d)$, is defined by $\mathsf{k}(x,y)v = \on{ev}_y (K \on{ev}_x^v)$.

Given a time-dependent vector field $u \in L^1([0,1],\mc H)$, it admits a flow, i.e., there exists a curve $\ph \in C([0,1],\on{Diff}^1_+(\R^d))$ solving
\begin{equation}
\label{eq:admissible_flow}
\p_t \ph(t) = u(t) \o \ph(t)\,,\qquad \ph(0) = \on{Id}\,,
\end{equation}
for $t \in [0,1]$ almost everywhere. 

We define the group $\mc G_H$ consisting of all flows that can be generated by $\mc H$-valued vector fields,
\[
\mc G_{\mc H} = \left\{ \ph(1) \,:\, \ph(t) \text{ is the solution of \eqref{eq:admissible_flow} with } u \in L^1([0,1], \mc H) \right\}\,.
\]
Then $\mc G_{\mc H} \subseteq \on{Diff}^1_+(\R^d)$ and one can show that $\mc G_{\mc H}$ is a group. We can define a distance on $\mc G_{\mc H}$ via
\begin{equation}
\label{eq:dist_G_H}
\on{dist}^{\mc H}(\ph, \ps) = \inf \left\{ 
\int_0^1 \| u(t) \|_{\mc H} \ud t \,:\, u \in L^1([0,1],\mc H),\, \ps = \on{Fl}_{1}(u) \o \ps \right\}\,.
\end{equation}
Then $(\mc G_{\mc H}, \on{dist}^{\mc H})$ is a complete metric space and the infimum in \eqref{eq:dist_G_H} is always attained; furthermore there always exist minima with $\| u(t)\|_{\mc H}$ constant in $t$. See \cite[Sect. 8]{Younes2010} for details and full proofs.

The space $\mc H$, where $\mathsf{k}$ is the Gaussian kernel
\[
\mathsf{k}(x,y) = \exp\left(-\tfrac{|x-y|^2}{\si^2}\right) \on{Id}_{d\x d}\,,
\] 
or a sum of Gaussian kernels is widely used for diffeomorphic image matching. For numerical reasons, the kernel associated with Sobolev spaces is used less.

Note that from an analytic point of view the class of admissible vector spaces is rather large. It contains finite-dimensional vector spaces as well as spaces on real-analytic vector fields; it makes no assumptions about the decay of the vector fields at infinity other than that they are bounded; any closed subspace of an admissible vector space is itself admissible. Therefore there are limits as to how far a general theory can be developed: $\mc G_H$ does not need to have a differentiable structure; $\mc G_{\mc H}$ with the topology induced by the metric $\on{dist}^{\mc H}$ does not need to be a topological group; there is no known natural topology on $\mc G_{\mc H}$ making it a topological group.

\subsection{Equivalence of groups}
The situation is more promising, if $\mc H$ is a Sobolev space. In this case we can use Thm.~\ref{thm:d2plus1_flow} to characterize the group generated by $\mc H$: the group $\mc G_{H^s}$ coincides with the connected component of the identity of the group of Sobolev diffeomorphisms.

\begin{theorem}
\label{thm:equivalence}
Let $s > d/2+1$. Then
\[
\mc G_{H^s(\R^d,\R^d)} = \mc D^s(\R^d)_0\,.
\]
\end{theorem}
\begin{proof}
Let $U$ be a convex neighborhood around $\on{Id}$ in $\mc D^s(\R^d)$. Then every $\ps \in U$ can be reached from $\on{Id}$ via the smooth path $\ph(t) = (1-t) \on{Id} + t \ps$. Since $\ph(t)$ is the flow of the associated vector field $u(t) = \p_t \ph(t) \o \ph(t)\i$ and $u \in C([0,1], H^s)$, it follows that $\ps \in \mc G_{H^s}$. Thus $U \subseteq \mc G_{H^s}$ and since $\mc G_{H^s}$ is a group, the same holds also for the whole connected component containing $U$. This shows the inclusion $\mc D^s(\R^d)_0 \subseteq \mc G_{\mc H}$.

For the inclusion $\mc G_{H^s} \subseteq \mc D^s(\R^d)$ we have to show that given a vector field $u \in L^1([0,1], H^s(\R^d,\R^d))$ the flow defined by \eqref{eq:admissible_flow} is a curve not only on $\on{Diff}^1_+(\R^d)$, but also in $\mc D^s(\R^d)$. This is the content of Thm. \ref{thm:d2plus1_flow}.
\end{proof}

So when $\mc H = H^s$ is a Sobolev space, then the group $\mc G_{H^s}$ is a smooth Hilbert manifold as well as a topological group. If additionally the right-invariant metric induced by the inner product on $H^s$ is smooth, then the distance defined in \eqref{eq:dist_G_H} coincides with the geodesic distance. In particular paths of minimal length are smooth in time.

\begin{openquestion*}
When $\mc H$ is a Sobolev space and the induced right-invariant metric is smooth on $\mc D^s(\R^s)$, the corresponding geodesic equation is called the EPDiff equation. In order to write the geodesic equation, one only needs the kernel $\mathsf k(\cdot,\cdot)$ and it would be of interest to study its solutions for those kernels, where the induced groups don't carry a smooth structure.
\end{openquestion*}

\subsection{Karcher means of images}

Diffeomorphic image matching solves the minimization problem \cite{Beg2005}
\begin{equation}
\mc J(\ph) = \frac 12  \on{dist}^s(\on{Id},\ph)^2 + S(I \o \ph\i , J)\,,
\end{equation}
where $I,J \in \mc F(\R^d,\R)$ are respectively the source image and the target image. The term $S$ measures the similarity between the deformed image $I \o \ph\i$ and $J$. Its simplest form is the $L^2$ distance between the two functions. Therefore, optimal paths are geodesics on $\mc G_{\mc H}$. At a formal level, the situation can be understood as follows: The composition $I\o \ph\i$ is a left action of the group of diffeomorphisms $\mc G_{\mc H}$ on the space of images. 
The strong Riemannian structure on the group of diffeomorphisms $ \mc D^s(\R^d)$ and its completeness enable the application of results showed using proximal calculus on Riemannian manifolds \cite{Azagra2005}.

\begin{proposition}
Let $I \in L^1(\R^d,\R)$ be an image and $\mc{O}_I$ its orbit under the action of $ \mc D^s(\R^d)$. 
There exists a dense set $D \subset \mc{O}_I^n$ such that if $ (I_1,\ldots,I_n) \in D$, then there exists a unique minimizer in $\mc O_{I}$ of 
\begin{equation}
 \sum_{k=1}^n d(J,I_k)^2 \,,
\end{equation}
where $d$ is the induced distance on the orbit $\mc{O}_I$ defined by 
\[
d(I,J) = \inf_{\ph \in \mc D^s(\R^d)} \left\{ \on{dist}^s(\on{Id},\ph) \, | \, I \circ \ph\i = J \right\}\,.
\] 
In other words, the Karcher mean of a set of images in $D$ is unique.
\end{proposition}

\begin{proof}
Since the action of $\mc D^s(\R^d)$ on $L^1(\R^d,\R)$ is continuous, the isotropy subgroup of $I$ denoted $\mc D_I$ is a closed subset of $\mc D^s(\R^d)$. 
Since each image $I_k$ lies in the orbit $\mc O_I$, there exist $\ph_k \in \mc D^s(\R^d)$, such that $I_k = I \o \ph_k\i$. 
Define
\[
C = \ph_1 \o \mc D_I \x \dots \x \ph_k \o \mc D_I
\]
Clearly, the set $C \subset \mc D^s(\R^d)^n$ is closed and nonempty.
Note that the product distance $\on{dist}^{s,n}$ on $\mc D^s(\R^d)^n$ derives from a smooth Riemannian metric with the property that any two points can be joined by a minimizing geodesic.
Using \cite[Thm. 3.5]{Azagra2005}, there exists a dense subset $D' \subset \mc D^s(\R^d)^n$ such that $\Phi \in \mc D^s(\R^d) \mapsto \on{dist}^{s,n}(\Phi,C)$ is differentiable at the points $\Phi \in D'$ and there exists a unique minimizing geodesic between $\Phi$ and $C$.
We have 
\begin{multline}
\on{dist}^{s,n}(\Phi, C)^2 = \inf_{\ph \in \mc D^s(\R^d)}\sum_{k=1}^n \on{dist}^{s}(\ph_k,\ph \, \mc D_I)^2  =  \inf_{\ph \in \mc D^s(\R^d)}\sum_{k=1}^n \on{dist}^{s}(\ph_k \mc D_I,\ph \, \mc D_I)^2 \\
 =  \inf_{\ph \in \mc D^s(\R^d)}\sum_{k=1}^n d(I \o \ph_k\i, I \o \ph\i)^2\,.
\end{multline}
Therefore, the image of $D'$ by action on $I$ gives the subset $D$ dense in $\mc O_I^n$. 
\end{proof}

This is a weak generalization of Ekeland's result \cite{Ekeland1978} on generic uniqueness of geodesics. 

\appendix

\section{Carathéodory Differential Equations}
\label{sec:caratheodory}

Let $I$ be an interval, $X$ a Banach space and $U \subseteq X$ an open subset of $X$. If $f:I \x U \to X$ is continuous and satisfies the Lipschitz condition
\[
\|f(t,x) - f(t,y)\|_X \leq L \| x-y\|_X
\]
for all $t \in I$ and $x,y \in U$, then the ODE
\begin{align*}
\p_t x(t) &= f(t,x(t)) \\
x(t_0) &= x_0\,,
\end{align*}
with $t_0 \in I$ and $x_0 \in U$ has a unique solution on some small interval $[t_0-\de, t_0+\de]$. This result is a straight-forward generalisation from ODEs in $\R^d$ and can be found in several books. See, e.g. \cite{Martin1976} or \cite{Deimling1977}.

To apply techniques from variational calculus it is convenient to work with vector fields $u \in L^2([0,1], \mc H)$ where $\mc H$ is a Hilbert space of $C^1_b$-vector fields on $\R^d$. The flow equation of these vector fields,
\[
\p_t \ph(t) = u(t) \o \ph(t)\,,
\]
leads to differential equations, whose right hand side is not continuous in $t$ any more, but only measurable. Such ODEs are called differential equations of {\it Carathéodory type}. Since Carathéodory differential equations might be unfamiliar to some readers, we will state here the results, that are used in this article. Following the exposition of \cite{Aulbach1996} we define:

\begin{definition}
\label{def:caratheodory}
Let $I$ be a nonempty interval, $X$ a Banach space and $U \subseteq X$ an open subset. A mapping $f : I \x U \to X$ is said to have the \emph{Carathéodory property} if it satisfies the following two conditions:
\begin{enumerate}
\item
For every $t \in I$ the mapping $f(t,\cdot) : U \to X$ is continuous.
\item
For every $x \in U$ the mapping $f(\cdot,x) : I \to X$ is strongly measurable (with respect to the Borel $\si$-algebras), i.e., $f(\cdot,x)$ is measurable and the image $f(I,x)$ is separable.
\end{enumerate}
\end{definition}

We have the following basic existence result for Carathéodory type differential equations.

\begin{theorem}
\label{thm:caratheodory_exist}
Given an interval $I=[a,b]$ and a Banach space $X$, let $U \subseteq X$ be an open subset and $f: I \x U \to X$ have the Carathéodory property. Given $x_0 \in U$ let $\ep$ be such that $B_\ep(x_0) = \{ x \,:\, |x-x_0|< \ep \} \subseteq U$. Furthermore let $m, \ell : I \to \R_{>0}$ be locally integrable functions such that the two estimates
\begin{align*}
\| f(t,x_1) - f(t,x_2)\|_X &\leq \ell(t)\, \| x_1 - x_2 \|_X \\
\| f(t,x) \|_X &\leq m(t)
\end{align*}
are valid for almost all $t \in I$ and for all $x,x_1,x_2 \in B_{\ep}(x_0)$. Finally let $\de > 0$ be such that
\begin{equation}
\label{eq:def_delta}
\int_a^{a+\de} m(t) \ud t < \ep\,.
\end{equation}
Then the differential equation
\[
\p_t x(t) = f(t,x(t))
\]
has a unique solution $\la : [a,a+\de] \to B_\ep(x_0)$ satisfying the initial condition $\la(a)=x_0$, i.e.
\[
\la(t) = x_0 + \int_a^t f(\ta, \la(\ta))\ud \ta
\]
holds for all $t \in [a, a+\de]$. The function $\la$ is absolutely continuous.
\end{theorem}

\begin{proof}
This is essentially \cite[Thm. 2.4]{Aulbach1996}. The condition \eqref{eq:def_delta} is taken from \cite[Thm. 1.1.1]{Filippov1988} to ensure that the mapping
\[
T(\mu)(t) := x_0 + \int_{a}^t f(\ta, \mu(\ta)) \ud \ta
\]
maps continuous functions $\mu : [a,a+\de) \to B_{\ep}(x_0)$ to continuous functions with values in $B_{\ep}(x_0)$. The rest of the proof in \cite{Aulbach1996} can be used without change.
\end{proof}

For linear equations it is enough that the right hand side be integrable. See \cite[p. 55f]{Aulbach1996}.

\begin{theorem}
\label{thm:caratheodory_linear}
Given an interval $I=[a,b]$, a Banach space $X$ and an element $x_0 \in X$, let $A: I \to L(X)$ and $b : I \to X$ be Bochner integrable functions, i.e. both functions are strongly measurable and the real-valued functions $\| A(\cdot)\|_{L(X)}$ and $\| b(\cdot)\|_X$ are integrable. Then the differential equation
\[
\p_t x(t) = A(t).x(t) + b(t)
\]
has a unique solution $\la : I \to X$ satisfying the initial condition $\la(a) = x_0$.
\end{theorem}

The theory of Carathéodory type differential equations can be found in \cite{Coddington1955} and \cite{Filippov1988} for $\dim X< \infty$ and in \cite{Aulbach1996}, \cite{Deimling1977} or \cite{Younes2010} for infinite-dimensional spaces.

\printbibliography

\end{document}